\documentclass[12pt,a4paper]{article}

\usepackage{comment}
\usepackage[colorlinks=true, urlcolor=blue,linkcolor=blue,citecolor=blue, pdfborder={0 0 0}]{hyperref}
\usepackage[left=2.50cm, right=2.50cm, bottom=3.00cm, top=3.00cm]{geometry}
\usepackage{amsthm, amsmath, amssymb, mathtools,yhmath}
\usepackage{thmtools}
\usepackage{enumitem} 
\usepackage{setspace}
\usepackage{lipsum}

\newcommand{\R}{{\mathbb{R}}}
\newcommand{\F}{{\mathcal{F}}}
\newcommand{\E}{{\mathbb{E}}}
\newcommand{\Z}{{\mathbb{Z}}}

\DeclareMathOperator\diam{diam}
\DeclareMathOperator\supp{supp}
\DeclareMathOperator\V{\mathbb{V}ar}
\DeclareMathOperator\bias{bias}
\DeclareMathOperator\tr{tr}
\DeclareMathOperator\cov{cov}
\DeclareMathOperator\diver{div}
\DeclareRobustCommand{\bbone}{\text{\usefont{U}{bbold}{m}{n}1}}

\newcounter{matematica}
\declaretheoremstyle[numbered=yes,
		spaceabove=\topsep, spacebelow=\topsep,
		headfont=\normalfont\bfseries,
		notefont=\mdseries, notebraces={(}{)},
		bodyfont=\itshape,
        postheadspace=.5em,
		]{mercteo}

\declaretheorem[name=Remark, style=mercteo, numberlike=matematica]{rem}
\declaretheorem[name=Definition, style=mercteo, numberlike=matematica]{defi}
\declaretheorem[name=Lemma, style=mercteo, numberlike=matematica]{lema}
\declaretheorem[name= Theorem, style=mercteo, numberlike=matematica]{teo}
\declaretheorem[name=Corollary, style=mercteo, numberlike=matematica]{cor}
\declaretheorem[name=Proposition, style=mercteo, numberlike=matematica]{prop}
\declaretheorem[name=Example, style=mercteo, numberlike=matematica]{ex}
\usepackage[round, sort, square]{natbib}
\setcitestyle{authoryear,open={(},close={)}} 

\title{Nonparametric inference on Fokker-Plank and McKean-Vlasov models}
\author{Adriana Laurindo Monteiro \thanks{Independent Researcher. This work was developed during the author's Phd at Instituto de Matem\'{a}tica Pura e Aplicada (IMPA), Rio de Janeiro, Brazil. \texttt{mlaurindodrica@gmail.com}.} and Roberto Imbuzeiro Oliveira\thanks{Instituto de Matem\'{a}tica Pura e Aplicada (IMPA), Rio de Janeiro, Brazil. \texttt{rimfo@impa.br}. Supported by a {\em Bolsa de Prdutividade em Pesquisa} from CNPq, Brazil (\# 305765/2023-0); an Excellent Science - Marie Skłodowska-Curie Actions grant from the European Commussion (DOI: 10.3030/101007705); and the following grants from Faperj, Rio de Janeiro, Brazil: {\em Cientista do Nosso Estado} (\# E-26/200.485/2023) and {\em Edital Inteligência Artificial} (\# E-26/290.024/2021). }}
\begin{document}

\maketitle

\begin{abstract}We propose a kernel-based estimator of the velocity field governing the transport and diffusion of d-dimensional interacting particles. Assuming the initial positions are i.i.d with law $\mu_0$, we establish consistency of the estimator with an explicit mean-squared error rate of $O(h^2 + N^{-2/(d+2)}))$, where $h$ denotes the time-discretization step and $N$ the number of particles. The analysis covers two distinct settings: the Fokker–Planck equation, where we recover the underlying potential function, and the McKean–Vlasov equation, where we deconvolve the interaction kernel driving the mean-field dynamics. \end{abstract}


				
		


\section{Introduction}
\subsection{Motivation and related work}
    Systems of interacting particles are of strong interest in applied mathematics. 
    These models can simulate the dynamics of various events in biology \citep{bio}, social sciences \citep{pop}, finance \citep{finance} and deep learning \citep{deep}. Two of the most important systems are described by famous partial differential equations: the Fokker-Planck and McKean-Vlasov.
    
    It is worth to note that Fokker-Planck equation \eqref{FPE} (FPE) and some others evolution partial differential equations (PDE) 

    \begin{equation}\label{FPE}
        \partial_t \mu_t + \diver\Big(\mu_t(-\nabla V_t-\nabla \log\mu_t)\Big) = 0
    \end{equation}
    can be viewed as a gradient flow in the Wasserstein space, \citep{santamb}. This connection between the theories of optimal transport and gradient flows is a key ingredient in the study  of  the dynamics of stochastic gradient descent method on neural networks \citep{gradient}. Here $\mu_t:\Pi^d\to \R$ denotes the density of a random particle's position and $V_t:\R^d\to \R$ is a potential function that defines the drift term $\nabla V_t$ and represents an external force in the motion of each particle.    
    
    A McKean-Vlasov model \citep{McKean}, which can be written as the following PDE

    \begin{equation}\label{MVE}
        \partial_t \mu_t + \diver\Big(\mu_t(-\nabla V_t + F \ast \mu_t- \nabla \log \mu_t)\Big) = 0,
    \end{equation}
    describes phenomena, for instance, in fluid mechanics \citep{fluid} and electromagnetism \citep{electro}. In this case, the function $F:\Pi^d\to\R^d$ denotes an interaction kernel force applied to the system.

    Besides using the modelling power of these PDE's, plenty of works have the main purpose of finding the solution $\mu_t$. Examples include numerical methods \citep{numerical}, variational methods \citep{variational} and more recently physics-informed neural-networks \citep{pinn}. All of them try to directly learn or approximate the density function.

    In this paper we focus in a different quantity. We are interested in estimate the velocity field that governs the evolution of the particle's positions. Our non-parametric framework gives results for both FPE and MVE models in a unified manner. Some works in this context of statistical inference for PDE's usually  focus on one model like \citep{laetita} and \citep{artigoFPE}. 

In \citep{artigoFPE} they explore a property called self consistency of the FPE to estimate its solution using a neural network. Recently, the same approach is applied in more general settings such as MVE in \citep{shen2023}. Although the work in \citep{laetita} is also based on kernel estimators, they do not discretize the time-interval and more importantly, they observe a stochastic system of interacting particles. 

Also in the stochastic setting, we can cite \citep{bandi}, \citep{asymp_ergodic} and \citep{nonpara_mult} that are related to our problem. The first proposes a Nadaraya-Watson kernel estimate for the vector field very similar to the one proposed here. Although they give asymptotic distribution and consistency, they do not provide convergence rate.
\citep{asymp_ergodic} builds a kernel-based estimator as well, based on an estimator of the density. A different approach for multidimensional inference is used in \citep{nonpara_mult}, where convergence rates are given for a penalized least squares estimator of the vector field.


    \subsection{Setting}
    
    Let $\Pi^d = \R^d/R\Z^d$ be a torus, the cube $[-R/2,R/2]^d$ with periodic boundary conditions for some $R>0$. Consider a particle evolving in $\Pi^d$ in the time interval $[0,T]$ according to a velocity field $v_t$ with random initial position distributed as $\mu_0$, an $R$-periodic density function defined in $\Pi^d$. Suppose the original position is $x$ and $f(x,t)$ denotes its position in time $t$, where $f:[0, T]\times \Pi^d\to \R^d$. The evolution in time of the transport phenomena governed by $v_t$ is described by the following ordinary differential equation (ODE):
    
    \begin{equation}\label{ode}
		    \begin{cases}
		            \frac{df(x, t)}{dt}=v_t(f(x,t))\\
		            f(x,0)=x.
		    \end{cases}
		\end{equation}
    With Lipschitz conditions on $v_t$, one can prove existence and unicity of a flow $X_t$ given by integration in time 
    \begin{equation}\label{flux}
        X_t(x) \coloneqq x + \int_0^t v_s(f(x,s))ds,
    \end{equation}
    where the sum is considered modulus $R$. Imposing $R$-periodicity on the velocity field $v$ for all $t\in[0,T]$ implies the same periodicity for $X_t(x) - x$ as pointed in \citep{shen2023}. Therefore, we can define for any $t\in[0, T]$ an $R$-periodic push-forward density $\mu_t \coloneqq (X_t)_{\#}\mu_0$ that is a weak solution of the so called continuity equation for  $v_t$:
    \begin{equation}\label{eqcont}
			\partial_t\mu_t+\diver(\mu_tv_t)=0.
			\end{equation}
   
   Note that the $\mu_t:[0, T]\times \Pi^d\to \R$ are time-varying densities defined in $\Pi^d$. As proved respectively in \citep{mastereq} and \citep{guillin2023uniformtimepropagationchaos} the FPE and MVE both have classical solution on the torus $\Pi^d$, in the sense that $\mu$ is a smooth function  in time and space variables satisfying \eqref{eqcont}. Besides that, as we are going to prove in section \eqref{sec:regest}, the densities are bounded from above and from below uniformly in time. Therefore, we can assume the existence of $\lambda_1,\lambda_2>0$ such that
   \begin{align}
       \lambda_1\leq \mu_t\leq \lambda_2 \text{ for all } t\in[0,T].
   \end{align}
   
   Our problem is to observe a deterministic system of $N$ independent and identically distributed interacting particles and try to estimate $v_t$. Specifically, we have for each $i=1, \ldots, N$ an ODE
\begin{equation}
		    \begin{cases}
		            \frac{df(x_i,t)}{dt}=v_t(f(x_i,t))\\
		            f(x_i,0)=x_i,
		    \end{cases}
		\end{equation}
   where, for a function $f:[0, T]\times \Pi^d\to \R^d$, $f(x_i,t)$ represents the position of a particle in time $t$ originally located at $x_i$. We will adopt the notation $X_t^i$ to denote $f(x_i,t)$, i.e., the position of a particle in time $t$ that was located at $x_i$ when $t=0.$ With this notation, the random vectors $(X_0^1,\ldots, X_0^N) \sim \mu_0^{\otimes N}$ are such that $(X_0^1,\ldots, X_0^N)=(x_1, \ldots, x_N)$ and the random functions $(X_t^1,\ldots, X_t^N) \sim \mu_t^{\otimes N}$, where $\mu_t$ is the push-forward density already defined.

   Our estimation method is based on given observations in specific instants of time. Considering a time-step\footnote{Let $t_0=0$ and $t_M=T$.} $h = T/M$, formally our data is 
   \begin{align}
       &X_{t_j}^1, \ldots, X_{t_j}^N, \text{ for } j = 0, 1, \ldots, M\\
       &(X_0^1,\ldots, X_0^N) \sim \mu_0^{\otimes N}.
   \end{align}
Note that the derivative in time of the position is given by $v_t$. The approximation 
\begin{eqnarray}\label{derivada}
\frac{X_{t_{j+1}}^i-X_{t_j}^i}{t_{j+1}-t_j}\approx v_t(X_{t_j}^i)
\end{eqnarray}
motivates the definition of a kernel-based estimator for $v_t$:
\begin{equation}\label{estv}
    v_\theta(t_j,x) =\frac{1}{Z_\theta}\sum_{i=1}^N\kappa_\theta(x)\frac{X_{t_{j+1}}^i-X_{t_j}^i}{t_{j+1}-t_j}, \text{ for } j = 0, \ldots, M-1
\end{equation}
where $\kappa_\theta$ and $Z_\theta$ are respectively a kernel function and a normalizing factor to be defined later and $\theta\in (0,1)$. For simplicity, we will drop the parameter $\theta$ and denote $v_\theta(t_j,x)$ by $\hat{v}_{t_j}(x)$.

Using the estimator for $v_t$ defined in \eqref{estv}, we consider estimators for our two models, FPE and MVE. In the Fokker Planck case we have:
\begin{equation}
  v(t,x) = -\nabla V_t(x)-\nabla \log\mu_t(x).
\end{equation}

Defining an estimator for the density, the goal is to estimate the term $\nabla V_t$:
\begin{equation}\label{estdrift}
  \widehat{\nabla V_{t_j}}=-\widehat{v}_{t_j}(x)-\widehat{\nabla\log\mu_{t_j}(x)}.  
\end{equation}

In the McKean-Vlasov model the velocity field is given by:
\begin{equation}
  v(t,x) = - \nabla V_t(x) + F \ast \mu_t(x) - \nabla \log \mu_t(x),  
\end{equation}
where the convolution is defined as $\int_{\Pi^d} F(x-y)\mu_t(y)dy$.
We will consider $V_t\equiv0$ and the goal is to estimate $F$ using deconvolution techniques:
\begin{equation}\label{estintker}
  \mathcal{F}[v_t + \nabla \log \mu_t] = \mathcal{F}[F]\cdot\mathcal{F}[\mu_{t}] \Longrightarrow \mathcal{F}[\widehat{F}]=\frac{\mathcal{F}[\hat{v}_t+\widehat{ \nabla \log \mu_t]}}{\hat{\mathcal{F}}[\mu_t]}.
\end{equation}

\subsection{Results and organization of the paper}
The main result is that our estimator is $O(h^2 + N^{-2/(d+2)})$ in mean-squared error. An advantage of the framework developed here over other approaches such as \citep{artigoFPE} and \citep{bandi} is that we can compute the rate of convergence.
\begin{teo}{(MSE OF $\hat{v}_{t_j}$; informal)}
    Fix $x\in\Pi^d$. For $\hat{v}_{t_j}$ defined in \eqref{estv}, there exists a constant $C$ independent of $N$ and $\theta$ such that
\begin{equation}
     \E\|\hat{v}_{t_j}(x)-v_{t_j}(x)\|^2\leq C(h^2 + N^{-2/(d+2)}).
\end{equation}
\end{teo}

In general, inference methods for vector fields are restricted to specific PDEs, such as \citep{laetita} and \citep{artigoFPE}. Since many of these equations can be recovered from the continuity equation, our result can be applied to different models in a straightforward manner, which is the content of the following two results.

\begin{teo}{(MSE DRIFT; informal)}
    Fix $x\in\Pi^d$. For $\widehat{\nabla V_{t_j}}(x)$ defined in \eqref{estdrift}, there exists a constant $C$ independent of $N$ and $\theta$ such that
\begin{equation}
    \E\|\widehat{\nabla V_{t_j}}(x)-\nabla V_{t_j}(x)\|^2\leq C(h^2 + N^{-6/(d+6)}).
\end{equation}
\end{teo}

 Although the deconvolution techniques are enough to get consistency, we do not establish a convergence rate for the McKean Vlasov case, like in \citep{laetita}. The estimator here is mainly based on \citep{Johannes_2009}. 
\begin{teo}{(Consistency of kernel-interaction estimator; informal)}
    The estimator $\hat{F}$ defined in \eqref{estintker} satisfies
    \begin{equation} 
        \E\|\hat{F}-F\|^2_{L^2(\Pi^d,\R^d )} = o(1) \text{ as } N\to\infty.
    \end{equation}
\end{teo}

In Section \ref{sec:consest} we give details on how the estimators are defined, assumptions, as well as notation. Section \ref{sec:maintheo} is devoted to the main theorem: we establish an upper bound on the mean squared error of $\hat{v}_{t_j}$ using properties of the kernel and the regularity estimates of $\mu_t$. Both Section \ref{sec:fpe} and \ref{sec:mve} are concerned with applying this result on respectively estimating the potential and interaction kernel functions.
All of these findings rely on strong properties of the densities $\mu_t$ given by the structure of the continuity equation. Norms of the derivatives up to order $2$ are bounded uniformly in time, as computations in Section \ref{sec:regest} show. Finally, Section \ref{sec:teclem} contains technical lemmas needed for calculations, such as concentration inequalities and inequalities involving moments of our kernel estimators.

\textbf{Notation}
The norms $\|\cdot\|$ and $\|\cdot\|_{op}$  denote respectively the euclidean norm in $\R^d$ and the spectral norm for matrices and tensors. $\|\cdot\|_\infty$ is the sup norm for a function and maximum norm elementwise for a matrix or vector. 
$\overline{B}[0,R] = \{x\in\R^d\mid \|x\|\leq R\}$ is the closed unit ball. 
We use $\{e_i\}_{i=1}^d$ to denote the canonical base of $\R^d$ and the symbols $\nabla$ and $\nabla^2$ respectively denote the first and second order spatial derivatives of a function.  The divergent operator is denoted by  $\diver H  = \sum_{i=1}^d\partial H^i/\partial x_i$ for $H:\R^d\to\R^d$ and the superscript $H^i$ is used to denote the $i-$th coordinate of a multidimensional function $H.$ The operations in $\Pi^d$ are considered modulus $R$. The integral of a function in $\Pi^d$ is computed with respect to the Lebesgue measure restricted to the cube $[-R/2,R/2]^d$, denoted by \textit{dx}. Specifically, $dx(A)=\int_{A}dx=\int_{A\cap[-R/2,R/2]^d}dLeb(x).$

\section{Construction of estimators and assumptions}\label{sec:consest}

\subsection{Kernel estimators for \texorpdfstring{$v_{t}$}{vt}}

Let $\kappa:\R^d\to [0,+\infty)$ be a kernel function with $\supp(\kappa) \subset [-R/2,R/2]^d$ that satisfies
\begin{align}\label{kerpro}
    \int\kappa(u)du=1, \int u\kappa(u)du=0, \int\kappa^2(u)du<\infty.
\end{align}
If we take the parameter $0<\theta<1$ and consider the function $y\mapsto\kappa_\theta(y) = \kappa(y/\theta)/\theta^d$, then $\supp(\kappa_\theta) \subset [-R\theta/2,R\theta/2]^d\subset [-R/2,R/2]^d$.
Consider the periodized version of the kernel $\kappa_\theta$ in the following way:
\begin{align}\label{periodized}
    \tilde{\kappa_\theta}(x) = \sum_{z\in\Z}\kappa_\theta(x-Rz), \text{ for all } x\in\R^d.
\end{align}
Note that the previous series converges since there is a unique non-zero element. Indeed, for each $x\in\R^d$ there is a unique integer such that $x-Rz\in\supp{\kappa_\theta}$. It is well-known that $R$-periodic functions in $\R^d$ can be identified with functions in the torus that satisfy the analogous properties \eqref{kerpro} in $\Pi^d$. From now on, we are going to consider this identification and use $\kappa_\theta$ to denote the function on torus identified with its periodized version \eqref{periodized}.

Suppose also that the kernel satisfies for all $y\in\Pi^d$
\begin{equation}\label{nablaker}
    \|\nabla \kappa(y)\|\leq C_1|\kappa(y)|.
\end{equation}

\begin{ex}
    Consider the Epanechnikov kernel supported on the interval $[-R/2,R/2]:$
    \begin{align}\label{epane}
        \kappa_1(x_1) = a_1(1-x_1^2)\bbone_{|x_1|\leq R/2}
    \end{align}
where $a_1$ is the normalizing constant. We can construct a d-dimensional kernel using the spherical version of \eqref{epane}:
\begin{align}
    \kappa_S(x) = c_S \kappa_1(x^\top x),
\end{align}
where $c_S$ is a normalizing constant. Choosing $R=2+2\sqrt{2}$, $\kappa_S$ satisfies \eqref{nablaker}.
\end{ex}
 Denoting $\frac{X_{t_{j+1}}^i-X_{t_j}^i}{t_{j+1}-t_j}$ by $Y_{j,i}$, define for $x\in\Pi^d$

\[\kappa^{i,j}_\theta(x) = \kappa_\theta\left(x-X_{t_j}^i\right)\] 
and 
 \[Z_\theta(t_j,x):=\frac{1}{\theta^d}\sum_{i=1}^N\kappa\left(\frac{x-X_{t_j}^i}{\theta}\right) = \sum_{i=1}^N\kappa^{i,j}_\theta(x).\] 

Any finite measure on $\R^d$ can also be periodized to obtain a measure on $\Pi^d$. If $\mu$ is a probability measure in $\R^d$, we can take
\begin{align}
    \nu(A) = \sum_{z\in\Z^d}\mu(A+Rz).
\end{align}
In the following, we are considering periodized versions of gaussians measures. We will add noise terms to each system and to each time-step of discretization:
\[\left\{\epsilon_{j,i}\right\}_{j,i=0,1}^{M-1,N}\overset{\text{i.i.d}}{\sim} N(0,\sigma^2I), \mbox{ and } \bar{\epsilon}_{j,i}\coloneqq\frac{\epsilon_{j+1,i}-\epsilon_{j,i}}{t_{j+1}-t_j}, \]
where $I$ is the identity matrix on $\Pi^d$ and $\epsilon_{j,i}$ are independent of $X_{t_j}^i$, for all $i\in\{1,\ldots, N\}$ and all $j\in\{0,\ldots,M-1\}$. We propose the following estimators for $j = 0, 1, \ldots, M-1:$
	\[\hat{v}_{t_j}(x)\coloneqq\begin{cases}
	        e_1, \text{ if }  Z_\theta(t_j,x)=0,\\
	        \frac{1}{Z_\theta(t_j,x)}\sum_{i=1}^N\kappa^{i,j}_\theta(x)(Y_{j,i}+\epsilon_{j,i}), \text{ otherwise}
	\end{cases}\]
	and
	
	\[\bar{v}_{t_j}(x)\coloneqq\begin{cases}
	        e_1, \text{ if }  Z_\theta(t_j,x)=0,\\
	        \frac{1}{Z_\theta(t_j,x)}\sum_{i=1}^N\kappa^{i,j}_\theta(x)(Y_{j,i}+\bar{\epsilon}_{j,i}), \text{ otherwise.}
	\end{cases}\]

 When $t=t_M$, we can define the estimators as the last time step:
 \begin{align}
     &\hat{v}_{t_M}(x) = \hat{v}_{t_{M-1}}(x)\\
     &\bar{v}_{t_M}(x) = \bar{v}_{t_{M-1}}(x).
 \end{align}
 
 \begin{rem}\label{rem:interpolation}Interpolation between time steps

 Our estimators are defined for $t = t_0, \ldots, t_{M-1}, t_M$. We can interpolate them for other time steps, when $t \in(t_j, t_{j+1})$ for some $j = 0, \ldots, M$  as 
 \begin{align*}
     \hat{f}_{t}(x)\coloneqq \hat{f}_{t_j}(x)+\frac{t-t_j}{t_{j+1}-t_j}(\hat{f}_{t_{j+1}}(x)-\hat{f}_{t_{j}}(x)).
 \end{align*}
    Consistency theorems together with Lipschitz continuity in time of $f_t$ guarantees an uniform approximation when $h\to0$ and $N\to\infty:$
    \begin{align*}
        \sup_{t\in[0,T]}\E|\hat{f}_{t}(x)-f_{t}(x)|^2\to0.
    \end{align*}
    Indeed, if $t\in(t_{j}, t_{j+1})$ for some $j=0, \ldots, M-1$ and writing $\lambda \coloneqq\frac{t-t_j}{t_{j+1}-t_j}$:
    \begin{align*}
        &\E|\hat{f}_{t}(x)-f_{t}(x)|^2 = \E|(1-\lambda)\hat{f}_{t_j}(x) + \lambda \hat{f}_{t_{j+1}}(x) - f_{t}(x)|^2 \\
        &\leq C(1-\lambda)^2\E|\hat{f}_{t_j}(x)- f_{t_{j}}(x)|^2+ C \lambda^2\E|\hat{f}_{t_{j+1}}(x)- f_{t_{j+1}}(x)|^2+ C|f_{t_{j}}(x) - f_{t}(x)|^2\\ & \;\;\;\;\;\;+C\lambda^2|f_{t_{j+1}}(x) - f_{t_{j}}(x)|^2\\
        &\leq O(h^2+N^{-2/d+2})+ 2CL^2h^2.
    \end{align*}
 \end{rem}

\subsection{Density and derivative density estimators}
Consider the following estimators of the density $\mu_{t_j}$ and of its derivative, respectively, for $j = 0, 1, \ldots, M$ and $x\in \Pi^d$:
\begin{equation}
    \hat{\mu}_{t_j}(x)\coloneqq\frac{1}{N\theta^d}\sum_{i=1}^N\kappa\left(\frac{x-X_{t_j}^i}{\theta}\right) = \frac{1}{N}\sum_{i=1}^N\kappa_\theta^{i,j}(x).  
\end{equation}
and
\begin{equation}
    \widehat{\nabla \mu}_{t_j}(x)\coloneqq\frac{1}{N\theta^{d+1}}\sum_{i=1}^N\nabla\kappa\left(\frac{x-X_{t_j}^i}{\theta}\right)=\frac{1}{N\theta}\sum_{i=1}^N\nabla\kappa_\theta^{i,j}(x).
\end{equation}
When $t \in(t_j, t_{j+1})$ for some $j = 0, \ldots, M$, we are going to linearly interpolate it between time-steps as we did in Remark \ref{rem:interpolation} and the uniform approximation follows by Remark \ref{rem:lipdens}.

\subsection{Deconvolution estimator}
For a function $f:\Pi^d\to\Pi^d$ with $f^i\in L^1(\Pi^d)$ and a measure $\nu$ on $\Pi^d$, consider respectively its Fourier transforms
\begin{align*}\label{foutrans}
    \mathcal{F}[f](m) = \left(\int_{\Pi^d} f^1(x)\exp(-2\pi im^\top x)dx, \ldots, \int_{\Pi^d} f^d(x)\exp(-2\pi im^\top x)dx\right), m \in \Z^d
\end{align*}
and  
\begin{align*}
    \mathcal{F}[\nu](m) = \int_{\Pi^d} \exp(-2\pi im^\top x)d\nu (x), m \in \Z^d.
\end{align*}  

The Fourier transform of the empirical measure given by $\mu_{t_j}$ is 
\begin{equation}
  \hat{\mathcal{F}}[\mu_{t_j}](u)\coloneqq\frac{1}{N}\sum_{i=1}^N\exp(-iu^\top X_{t_j}^i), \text{ for } u\in\Z^d. 
\end{equation}
Let $\ell_s(u)\coloneqq(1+\|u\|^2)^\frac{s}{2}$, for $s\geq0$ and $A_s = \left\{u\in\Pi^d\mid\left|\frac{\hat{\F}[\mu_{t_j}](u)}{\ell_s(u)}\right|^2\geq\alpha\right\}$ for $\alpha>0$. 
The estimator of the interaction kernel is going to be defined using a regularized version of its Fourier transform:  
\begin{equation}
    \F[\widehat{F^i}]_s\coloneqq \frac{\F[\hat{v}_{t_j} + \widehat{\nabla\log\mu_{t_j}}] \cdot\overline{\hat{\F}[\mu_{t_j}]}}{|\hat{\F}[\mu_{t_j}]|^2}\cdot\bbone_{A_s}.
\end{equation}

{\color{red} When we apply this estimator in Section \ref{sec:mve}, we will need conditions that guarantee that the denominator in this expression does not vanish.}

\subsection{Assumptions}
We work under some smoothness assumptions on the velocity field and on the initial density $\mu_0.$

\subsubsection{Assumption 1}\label{assump:1}
The vector field is Lipschitz in space and time with constant $L$ (also called bi-Lipschitz):
\begin{equation}
    \|v_t(x)-v_k(y)\|\leq L\cdot\left(|t-k|+\|x-y\|\right).
\end{equation}

This assumption guarantees Lipschitz continuity in space uniformly in time. Then we get existence and uniqueness of global solution of the ODE and the flow $X_t^i$ is well-defined. 

\subsubsection{Assumption 2}\label{assump:2}
Also, for regularity results of the density we need the velocity field to satisfy for any space-time variables $x\in\Pi^d$ and $t\in[0,T]$
\begin{equation}
    \max\left\{\|\nabla\diver v_t(x)\|,\|\nabla v_t(x)\|_{op}\right\}\leq L_v
\end{equation}
and the initial density needs to satisfy for all $x\in\Pi^d$
\begin{equation}
    \|\nabla\log\mu_0(x)\|\leq C.
\end{equation}

\subsubsection{Assumption 3}\label{assump:3}
For high-order estimates, we need boundedness of high-order derivatives. Suppose that for any space-time variables $x\in\Pi^d$ and $t\in[0,T]$
\begin{equation}
    \max\left\{\|\nabla^2 \diver v_t(x)\|_{op}, \|\nabla^2 v_t(x)\|_{op},\right\}\leq L_v.
\end{equation}
and the initial density needs to satisfy for all $x\in\Pi^d$
\begin{equation}
    \|\nabla^2\log\mu_0(x)\|\leq C.
\end{equation}

\begin{rem}\label{rem:lipdens}
    The assumptions on the vector field $v_t$ and the density $\mu_t$ imply Lipschitz continuity of $\mu_t$ in time. Indeed, for $t_1<t_2$ in $[0,T]$ and $x\in\Pi^d$ fixed
\begin{align}
    |\mu(t_1,x)-\mu(t_2,x)|&\leq \int_{t_1}^{t_2}|\partial_s\mu(s,x)|ds\\
    &= \int_{t_1}^{t_2}|\mu(s,x)\diver v_s(x) + \langle\nabla\mu(t,x),v_s(x)\rangle|ds\\
    &\leq |t_1-t_2|\cdot\sup_{t,x}|\diver v_s(x)|\cdot|\mu(t,x)| +\sup_{t,x} \|\nabla\mu(t,x)\|\cdot\|v_s(x)\|\\
    &\leq C |t_1-t_2|.
\end{align}

\end{rem}
\section{Estimation of the vector field}\label{sec:maintheo}

Our main theorem explores the structure given by the continuity equation and the kernel estimators to obtain an explicit rate of convergence of $\hat{v}_{t_j}$ in mean squared error.

\begin{teo}{(MSE OF $\hat{v}_{t_j}$)}\label{teo:msevhat}
    Fix $x\in\Pi^d$. Assume that Assumptions 1 and 2 holds, $h=t_{j+1}-t_j$ is the size of time discretization and the kernel $\kappa$ satisfies $\supp(\kappa) \subset [-R/2,R/2]^d$. 
Additionally,  consider $0<\theta\leq\sqrt{\lambda_1}/\sqrt{2} $.
Then there exists a constant $C$ independent of $N$ and $\theta$ such that

\begin{equation}\label{cota_v_barra}
     \E\|\hat{v}_{t_j}(x)-v_{t_j}(x)\|^2\leq C\exp\left(-N C_{\theta,\kappa}\right)
     +16(LhC)^2+(R\theta L)^2+ \frac{8\sigma^2C}{N\theta^dR^{2d}}.
\end{equation}
The mean squared error of $\bar{v}_{t_j}$ is upper bounded by:
	\begin{equation}
    \E\|\bar{v}_{t_j}(x)-v_{t_j}(x)\|^2\leq C\exp\left(-N C_{\theta,\kappa}\right)+16(LhC)^2+(R\theta L)^2+ \frac{16\sigma^2}{h^2N} \frac{C}{\theta^dR^{2d}}.
\end{equation}
\end{teo}
The optimal theta for this bound is 

\[\theta_{opt} =CN^{-1/d+2},\]
which gives
\begin{align*}
    MSE(\hat{v}_{t_j}(x))\leq C\exp(-C N^{\frac{2}{d+2}})+16(LhC)^2+4(RL)^2\frac{1}{N^{\frac{2}{d+2}}}+ \frac{8\sigma^2C}{R^{2d}}\frac{1}{N^{\frac{2}{d+2}}}.
\end{align*}

If the parameter $\theta$ is not small enough, i.e., if $\theta\gtrsim\sqrt{\lambda_1}$, it can be shown that an analogue of the theorem still holds with a looser bound:
    
\begin{equation}
    \E\|\hat{v}_{t_j}(x)-v_{t_j}(x)\|^2\leq C+16(LhC)^2+(R\theta L)^2+ \frac{8\sigma^2C}{N\theta^dR^{2d}}.
\end{equation}

\begin{rem}
    The constant $C_{\theta,\kappa}$ is explicitly computed in Lemma \ref{lem:smallprob}, where one can check that it depends only on $\theta$ and on the boundedness conditions of $\kappa.$
\end{rem}
\begin{proof}
For fixed $x\in\Pi^d,$ define the set $E = \left\{Z_\theta(t_j,x)> \frac{\E Z_\theta(t_j,x)}{2}\right\}$. It is straightforward that $E\subset \{Z_\theta\neq 0\}$. Then the mean squared error is bounded by:

\begin{align*}\label{mse3}
 &\E \|\hat{v}_{t_j}(x)-v_{t_j}(x)\|^2\cdot\bbone_{\{Z_\theta = 0\}}+\E \|\hat{v}_{t_j}(x)-v_{t_j}(x)\|^2\cdot\bbone_{\{Z_\theta > 0\}\cap E^C}
 +\E \|\hat{v}_{t_j}(x)-v_{t_j}(x)\|^2\cdot\bbone_{\{Z_\theta > 0\}\cap E}.
\end{align*}
The upper bounds on $\{Z_\theta = 0\}$ and $\{Z_\theta > 0\}\cap E^C\coloneqq A$ are the same for $\hat{v}_{t_j}$ and $\bar{v}_{t_j}$. These are obtained by Lemma \ref{lem:smallprob} and boundedness of the vector field (implied by boundedness of $\Pi^d$). Indeed,
	\begin{eqnarray}
	\E\|\hat{v}_{t_j}(x)-v_{t_j}(x)\|^2\cdot\bbone_{\{Z_\theta = 0\}}
 &\leq&\mathbb{P}(E^C)\cdot\left(2+2\sup_{t,x}\|v_{t_j}(x)\|^2\right)\\
    &\leq& C\exp\left(-N C_{\theta,\kappa}\right).
	\end{eqnarray}

And for the case in $A,$ note that
\begin{align*}
    \|Y_{j,i}\|\leq\frac{1}{t_{j+1}-t_j}\int_{t_j}^{t_{j+1}}\|v_s(X_s^i))\|ds\leq \sup_{s,x}\|v_s(x)\|.
\end{align*} 
Then
\begin{align*}
\|\hat{v}_{t_j}(x)-v_{t_j}(x)\|\cdot\bbone_A\leq \bbone_A 2 \sup_{s,x}\|v_s(x)\|+\bbone_A\left\|\frac{1}{Z_\theta(t_j,x)}\sum_{i=1}^N\kappa^{i,j}_\theta(x)\epsilon_{j,i}\right\|.
\end{align*}
Taking the square and using $(a+b)^2\leq2a^2+2b^2$,
\begin{align*}
    \|\hat{v}_{t_j}(x)-v_{t_j}(x)\|^2\cdot\bbone_A&\leq \bbone_A 8 (\sup_{s,x}\|v_s(x)\|)^2+2\bbone_A\left\|\frac{1}{Z_\theta(t_j,x)}\sum_{i=1}^N\kappa^{i,j}_\theta(x)\epsilon_{j,i}\right\|^2.
\end{align*}
Since $\|.\|^2$ is a convex function, $\left\|\frac{1}{Z_\theta(t_j,x)}\sum_{i=1}^N\kappa^{i,j}_\theta(x)\epsilon_{j,i}\right\|^2\leq\frac{1}{Z_\theta(t_j,x)}\sum_{i=1}^N\kappa^{i,j}_\theta(x)\|\epsilon_{j,i}\|^2$ and taking expectation we get:

\begin{align*}
\E \|\hat{v}_{t_j}(x)-v_{t_j}(x)\|^2\cdot\bbone_A&\leq C \mathbb{P}(A) + \E\left(\frac{\bbone_A}{Z_\theta(t_j,x)}\sum_{i=1}^N\kappa^{i,j}_\theta(x)\|\epsilon_{j,i}\|^2\right)
\end{align*}

If $\mathcal{F}^N$ is the $\sigma-$ algebra generated by $X_{j,1}, \ldots, X_{j,N}$,
\begin{align*}
    \E\left(\frac{\bbone_A}{Z_\theta(t_j,x)}\sum_{i=1}^N\kappa^{i,j}_\theta(x)\|\epsilon_{j,i}\|^2\right)&=\E\left[\E\frac{\bbone_A}{Z_\theta(t_j,x)}\sum_{i=1}^N\kappa^{i,j}_\theta(x)\|\epsilon_{j,i}\|^2\biggl\vert\mathcal{F}^N\right]\\
    &=\E\|\epsilon_{j,i}\|^2\cdot\E\left[\E\frac{\bbone_A}{Z_\theta(t_j,x)}\sum_{i=1}^N\kappa^{i,j}_\theta(x)\biggl\vert\mathcal{F}^N\right]\\
    &=\sigma^2\cdot d\cdot \mathbb{P}(A)
\end{align*}

Therefore,
\begin{align*}
\E \|\hat{v}_{t_j}(x)-v_{t_j}(x)\|^2\cdot\bbone_A&\leq C \mathbb{P}(A) +\sigma^2d\mathbb{P}(A)\\
&\leq C \mathbb{P}(E^C)\\
&\leq C\exp\left(-N C_{\theta,\kappa}\right).
\end{align*}

Finally, the analysis on the set $\tilde{E}\coloneqq\{Z_\theta > 0\}\cap E$ relies on Lemma \ref{lem:delta}:

\begin{align*}
&\|\hat{v}_{t_j}(x)-v_{t_j}(x)\|\cdot\bbone_{\tilde{E}}\leq\\
&\leq\frac{\bbone_{\tilde{E}}}{Z_\theta(t_j,x)}\left(\sum_{i=1}^N\kappa^{i,j}_\theta(x)\|Y_{j,i}-v_{t_j}(X_{t_j}^i)\|+\left\|\sum_{i=1}^N\kappa^{i,j}_\theta(x)\epsilon_{j,i}\right\|+\sum_{i=1}^N\kappa^{i,j}_\theta(x)\|v_{t_j}(x)-v_{t_j}(X_{t_j}^i)\|\right)\\
&\lesssim\frac{\bbone_{\tilde{E}}}{Z_\theta(t_j,x)}\left(\sum_{i=1}^N\kappa^{i,j}_\theta(x)\|\delta(h)\|+\left\|\sum_{i=1}^N\kappa^{i,j}_\theta(x)\epsilon_{j,i}\right\|+\sum_{i=1}^N\kappa^{i,j}_\theta(x)\|x-X_{t_j}^i\|_\infty\cdot L\right)\\
&\lesssim\bbone_{\tilde{E}}2LhC+\frac{\bbone_{\tilde{E}}}{Z_\theta(t_j,x)}\left\|\sum_{i=1}^N\kappa^{i,j}_\theta(x)\epsilon_{j,i}\right\|+\bbone_{\tilde{E}} \frac{R}{2}\theta L.
\end{align*}
            
Using the inequality $(a+b)^2\leq 2a^2+2b^2$ and taking expectation,
\begin{eqnarray*}
\E\|\hat{v}_{t_j}(x)-v_{t_j}(x)\|^2\cdot\bbone_{\tilde{E}}\lesssim2(2(2LhC)^2+2(R/2\theta L)^2)+2\E\left[\frac{\bbone_{\tilde{E}}}{Z_\theta^2(t_j,x)}\left\|\sum_{i=1}^N\kappa^{i,j}_\theta(x)\epsilon_{j,i}\right\|^2\right].
\end{eqnarray*}
	
	let us calculate the last term separately:
	\begin{eqnarray}
	\E\left[\frac{\bbone_{\tilde{E}}}{Z_\theta^2(t_j,x)}\left\|\sum_{i=1}^N\kappa^{i,j}_\theta(x)\epsilon_{j,i}\right\|^2\right]
	&=&\sum_{i,l=1}^N\E\left[\frac{\bbone_{\tilde{E}}\kappa^{i,j}_\theta(x)\kappa^{l,j}_\theta(x)}{Z^2_\theta(t_j,x)} \right]\E\langle \epsilon_{j,i},\epsilon_{j,l}\rangle\\
	&=&N\E\left[\frac{\bbone_{\tilde{E}}{\kappa^{i,j}}^2_\theta(x)}{Z^2_\theta(t_j,x)} \right]\sigma^2\\
	&<&N\frac{4\E{\kappa^{i,j}}^2_\theta(x)}{N^2(\E\kappa^{i,j}_\theta(x))^2}\sigma^2.
	\end{eqnarray}
	
	For $\bar{v}_{t_j}$, the calculations are similar:
	
	\begin{align}
	\E\left[\frac{\bbone_{\tilde{E}}}{Z_\theta^2(t_j,x)}\left\|\sum_{i=1}^N\kappa^{i,j}_\theta(x)\bar{\epsilon}_{j,i}\right\|^2\right]
	&= \frac{N}{h^2}\E\left[\frac{\bbone_{\tilde{E}}{\kappa^{i,j}}^2_\theta(x)}{Z^2_\theta(t_j,x)} \right](2\sigma^2)\\
	&< \frac{N}{h^2}\frac{4\E{\kappa^{i,j}}^2_\theta(x)}{N^2(\E\kappa^{i,j}_\theta(x))^2}(2\sigma^2).
    \end{align}

 Using Lemma \ref{lem:boundkappa}, we get
 
	\begin{eqnarray}
	\E\|\hat{v}_{t_j}(x)-v_{t_j}(x)\|^2\cdot\bbone_{\tilde{E}}
	&\leq&16(LhC)^2+(R\theta L)^2+\frac{8\E{\kappa^{i,j}}^2_\theta(x)}{N(\E\kappa^{i,j}_\theta(x))^2}\sigma^2\\
    &\leq& 16(LhC)^2+(R\theta L)^2+ \frac{8\sigma^2}{N} \frac{\lambda_2\|\kappa\|_\infty}{\theta^d}\frac{2^{2d}}{\lambda_1^2C_4^2R^{2d}}\\
    &=& 16(LhC)^2+(R\theta L)^2+ \frac{8\sigma^2C}{N\theta^dR^{2d}}.
	\end{eqnarray}

 For $\bar{v}_{t_j}$,
\begin{eqnarray}
	\E\|\bar{v}_{t_j}(x)-v_{t_j}(x)\|^2\cdot\bbone_{\tilde{E}}
	&\leq&16(LhC)^2+(R\theta L)^2+2\frac{N}{h^2}\frac{4\E{\kappa^{i,j}}^2_\theta(x)}{N^2(\E\kappa^{i,j}_\theta(x))^2}(2\sigma^2)\\
    &\leq& 16(LhC)^2+(R\theta L)^2+ \frac{16\sigma^2}{h^2N} \frac{\lambda_2\|\kappa\|_\infty}{\theta^d}\frac{2^{2d}}{\lambda_1^2C_4^2R^{2d}}\\
    &=& 16(LhC)^2+(R\theta L)^2+ \frac{16\sigma^2}{h^2N} \frac{C}{\theta^dR^{2d}}.
	\end{eqnarray}

	In the end, the mean squared error of this estimator is upper bounded by:
\begin{equation}
    \E\|\hat{v}_{t_j}(x)-v_{t_j}(x)\|^2\leq C\exp\left(-N C_{\theta,\kappa}\right)+16(LhC)^2+(R\theta L)^2+ \frac{8\sigma^2C}{N\theta^dR^{2d}}.
\end{equation}
	Analogously, the mean squared error of $\bar{v}_{t_j}$ is upper bounded by:
	\begin{equation}
    \E\|\bar{v}_{t_j}(x)-v_{t_j}(x)\|^2\leq C\exp\left(-N C_{\theta,\kappa}\right)+16(LhC)^2+(R\theta L)^2+ \frac{16\sigma^2}{h^2N} \frac{C}{\theta^dR^{2d}}.
\end{equation}
\end{proof}

An important property imposed by the ODE \eqref{ode} is that the Lipschitz continuity of the vector field controls the derivative in time of the particle's positions. 
\begin{lema}\label{lem:delta}
Define the function $\delta:(0,+\infty)\to\R^d$ as follows
\[\delta(h)\coloneqq\frac{X(t_j+h,x)-X(t_j,x)}{h}-v_{t_j}(X(t_j, x)),\]
where $X(t,x)$ is the flux defined by equation \eqref{flux}. Suppose that Assumption 1 holds, then 
\begin{eqnarray}
	\|\delta(h)\|\leq 2LCh.
	\end{eqnarray}
 \end{lema}

\begin{proof}
Integrating in time \eqref{ode}, we get
 
    \begin{eqnarray*}
	\frac{X(t_j+h,x)-X(t_j,x)}{h}&=&\frac{1}{h}\int_{t_j}^{t_j+h}v_s(X(s, x))ds.
	\end{eqnarray*}
Rewriting the function $\delta$ as 
	\[\delta(h)=\frac{1}{h}\int_{t_j}^{t_j+h}v_s(X(s, x))-v_{t_j}(X(t_j, x))ds\]
and using the Lipschitz continuity of $v$ we have 
	\begin{eqnarray}
	\|\delta(h)\|
	&\leq&\frac{1}{h}\int_{t_j}^{t_j+h}\left\|v_s(X(s, x))-v_{t_j}(X(t_j, x))\right\|ds\\
	&\leq& L\cdot(h+\sup_{s\in[t_j,t_j+h]}\left\|X(s, x)-X(t_j, x)\right\|)\\
    &\leq&L\cdot\left(h+\sup_{s\in[t_j,t_j+h]}\left\|h\cdot \frac{X(s, x)-X(t_j, x)}{h}\right\|\right)\\
    &\leq&L\cdot h \left(1+\sup_{s\in[t_j,t_j+h]}\left\| \frac{X(s, x)-X(t_j, x)}{h}\right\|\right).
    \end{eqnarray}
    Write $s = t_j+r$ for $r\in[0,h]$. Then $1/h\leq1/r:$ 
    \begin{eqnarray}
    \|\delta(h)\|
    &\leq& L\cdot h\left(1+\sup_{r\in[0,h]}\left\|\frac{X(t_j+r, x)-X(t_j, x)}{r} \right\|\right)\\
	&=& L\cdot h\left(1+\sup_{r\in[0,h]}\left\|\delta(r)+v_{t_j}(X(t_j,x)) \right\|\right)\\
	&\leq& L\cdot h\left(1+\sup_{r\in[0,h]}\left\|\delta(r)\|+\|v_{t_j}(X(t_j,x)) \right\|\right).
	\end{eqnarray}

	Let $C=\|v_{t_j}(X(t_j,x))\|+1$ and $M>0.$ Define
 \begin{equation}
     h_0(M)\coloneqq\inf\left\{h>0\mid Mh\leq\|\delta(h)\| \right\}.
 \end{equation}
 If $r\leq h_0(M)$, by definition $\|\delta(r)\|\leq M r$, then
 \begin{equation}
     \sup_{r\in[0,h_0(M)]}\|\delta(h)\|\leq M r \leq M h_0(M)
 \end{equation} 
and
\begin{align}
M h_0(M)\leq \|\delta(h_0(M)\|&\leq L\cdot h_0(M)\left(C+ M h_0(M)\right)\\
&= L C h_0(M) + L M h_0^2(M).
\end{align}
Take $M = 2LC$ and suppose $h_0(M)\neq0$,
\begin{align}
    &h_0(M)(M - L C) \leq L M h_0^2(M) \\
    &\Longrightarrow LC \leq \frac{2 L^2C h_0^2(M)}{h_0(M)}\\
    &\Longrightarrow \frac{1}{2L} \leq h_0(M).
\end{align}
Consequently,  $h\in[0,\frac{1}{2L}]\subset[0,h_0(M)]$,
	\begin{eqnarray}
	\|\delta(h)\|\leq 2LCh.
	\end{eqnarray}
\end{proof}

\section{Potential estimator: Fokker-Planck case}\label{sec:fpe}

Depending on the vector field $v_t$ in \eqref{eqcont}, we can end up with different PDE's. An important one is Fokker-Planck Equation (FPE):
\[ \partial_t\mu_t+\diver(\mu_t(-\nabla V_t))=\Delta\mu_t,\]
obtained when 
\begin{equation}\label{vectFPE}
    v_t(x)= -\nabla V_t(x)-\nabla \log\mu_t(x).
\end{equation}

Using our estimator for $v_t$ and the expression \eqref{vectFPE}, we can now build an estimator for the potential term $\nabla V_t$:
\[\widehat{\nabla V_t}(x)=-\widehat{v}_t(x)-\widehat{\nabla \log\mu_t(x)}.\]

In order to bound MSE of $\widehat{\nabla V_t}$, note that 
\begin{align}
    \|\widehat{\nabla V_t}(x)-\nabla V_t(x)\|=\|-\widehat{v}_t(x)-\widehat{\nabla \log\mu_t(x)}-(-v_t(x)-\nabla \log\mu_t(x))\|\\
    \leq\|\widehat{v}_t(x)-v_t(x)\|+\|\widehat{\nabla \log\mu_t(x)}-\nabla \log\mu_t(x)\|.
\end{align}

Again using $\|a+b\|^2\leq 2\|a\|^2+2\|b\|^2$, 
\begin{align}
    \E\|\widehat{\nabla V_t}(x)-\nabla V_t(x)\|^2
    \leq2\E\|\widehat{v}_t(x)-v_t(x)\|^2+2\E\|\widehat{\nabla \log\mu_t(x)}-\nabla \log\mu_t(x)\|^2.
\end{align}
Since we already have bounded the first term, it remains to do the same for the second one.

To do this, define the following estimator of the derivative of log-density for $j=0, \ldots, M:$
\begin{align}\label{logdenhat}
\widehat{\nabla \log\mu_{t_j}(x)}=\widehat{\left(\frac{\nabla \mu_{t_j}(x)}{\mu_{t_j}(x)}\right)}\coloneqq\begin{cases}
	        0, \text{ if }  Z_\theta(t_j,x)=0,\\
	        \frac{1}{\theta}\frac{\sum_{i=1}^N\nabla\kappa\left(\frac{x-X_{t_j}^i}{\theta}\right)}{\sum_{i=1}^N\kappa\left(\frac{x-X_{t_j}^i}{\theta}\right)}, \text{ otherwise.}
	\end{cases}
\end{align}

When $t \in(t_j, t_{j+1})$ for some $j = 0, \ldots, M$, we are going to linearly interpolate it between time-steps as we did in Remark \ref{rem:interpolation} and the uniform approximation also holds by Remark \ref{lem:timedergradlogmu}.

The idea is to work with the numerator and denominator separately. This section is mainly based in Chapter 3 of \citep{Garcia-Portugues2024}.

\subsection{Kernel density estimator}
Consider the following estimator of the density $\mu_{t_j}$, for $j=0,\ldots, M$:
\begin{equation}
    \hat{\mu}_{t_j}(x)\coloneqq\frac{1}{N}\sum_{i=1}^N\kappa_\theta^{i,j}(x).
\end{equation}

When $t \in(t_j, t_{j+1})$ for some $j = 0, \ldots, M$, we are going to linearly interpolate it between time-steps as we did in Remark \ref{rem:interpolation}.

\begin{lema}(MSE kernel density estimator)\label{lem:dens}
Suppose that $\int_{\Pi^d} u\kappa(u)du=0$ and $\eta=\int_{\Pi^d}\kappa^2(u)du<\infty.$ Then
\begin{align}
    \E\big[\hat{\mu}_{t_j}(x)-\mu_{t_j}(x)\big]^2\leq\frac{C_2^2\theta^4}{4}\|\nabla^2\mu_{t_j}\|_\infty^2+\frac{\|\mu_{t_j}\|_\infty}{N\theta^d}\eta, \text{ for } j=0,\ldots, M.
\end{align} 
\end{lema}
\begin{proof}

    First we need to bound the bias term:
\begin{align}
    \bias\hat{\mu}_{t_j}(x)=\E\hat{\mu}_{t_j}(x)-\mu_{t_j}(x)&=\frac{1}{\theta^d}\int_{\Pi^d}\kappa\left(\frac{x-y}{\theta}\right)\mu_{t_j}(y)dy-\mu_{t_j}(x)\\
    &=\int_{\Pi^d}\kappa(u)\mu_{t_j}(x-\theta u)du-\mu_{t_j}(x)\\
    &=\int_{\Pi^d}\kappa(u)\left(\mu_{t_j}(x-\theta u)-\mu_{t_j}(x)\right)du.
\end{align}
By the Taylor expansion of $\mu_{t_j}$ and using the mean zero assumption on the kernel $\kappa$, there exists a constant $C_2$ such that

\begin{align}
    \E\hat{\mu}_{t_j}(x)-\mu_{t_j}(x)\leq\frac{C_2\theta^2}{2}\|\nabla^2\mu_{t_j}\|_\infty.
\end{align}
For the variance, we have:
\begin{align}\label{boundvar}
    \V\kappa\left(\frac{x-X_{t_j}^i}{\theta}\right)\leq\E\kappa^2\left(\frac{x-X_{t_j}^i}{\theta}\right)&=\int_{\Pi^d}\kappa^2\left(\frac{x-y}{\theta}\right)\mu_{t_j}(y)dy\\
    &=\theta^d\int_{\Pi^d}\kappa^2(u)\mu_{t_j}(x-\theta u)du\\
    &\leq\theta^d\|\mu_{t_j}\|_\infty\int_{\Pi^d}\kappa^2(u)du
\end{align}
so 

\begin{align}
    \V\hat{\mu}_{t_j}(x)&=\frac{N}{N^2\theta^{2d}}\V\kappa\left(\frac{x-X_{t_j}^i}{\theta}\right)\\
    &\leq\frac{\|\mu_{t_j}\|_\infty}{N\theta^d}\eta.
\end{align}

Since $\E\big[\hat{\mu}_{t_j}(x)-\mu_{t_j}(x)\big]^2=\bias^2\hat{\mu}_{t_j}(x)+\V\hat{\mu}_{t_j}(x)$,
\begin{align}
    \E\big[\hat{\mu}_{t_j}(x)-\mu_{t_j}(x)\big]^2\leq\frac{C_2^2\theta^4}{4}\|\nabla^2\mu_{t_j}\|_\infty^2+\frac{\|\mu_{t_j}\|_\infty}{N\theta^d}\eta.
\end{align}
\end{proof}

\subsection{Kernel derivative density estimator}

Consider the following estimator of the gradient of the density $\mu_{t_j}$:
\begin{equation}\label{nablamu}
    \widehat{\nabla \mu}_t(x)\coloneqq\frac{1}{N\theta}\sum_{i=1}^N\nabla\kappa_\theta^{i,j}(x), \text{ for } j=0, \ldots, M.
\end{equation}

When $t \in(t_j, t_{j+1})$ for some $j = 0, \ldots, M$, we are going to linearly interpolate it between time-steps as we did in Remark \ref{rem:interpolation}. Lipschitz continuity of \eqref{nablamu} is obtained by Lemma \ref{lem:timedergradlogmu} and Remark \ref{rem:lipdens}.

\begin{lema}(MSE kernel derivative density estimator)\label{lem:derdens}
Suppose the kernel is supported in the torus $\Pi^d$.
  
Then
\begin{align}
    \E\|\widehat{\nabla \mu}_{t_j}(x)-\nabla \mu_{t_j}(x)\|^2\leq R^2\theta^2\|\nabla^2\mu_{t_j}\|_\infty^2+\frac{\|\mu_{t_j}\|_\infty}{N\theta^{d+2}}\eta.
\end{align}
\end{lema}

\begin{proof}

The mean vector can be written using the partial derivatives with respect to $y$, instead of $x$:
\begin{align}
    \E\widehat{\nabla \mu}_{t_j}(x)
    &=\frac{-1}{\theta^{d+1}}\left(\int\frac{\partial}{\partial y_1}\kappa\left(\frac{x-y}{\theta}\right)\mu_{t_j}(y)dy, \ldots, \int\frac{\partial}{\partial y_d}\kappa\left(\frac{x-y}{\theta}\right)\mu_{t_j}(y)dy\right)
\end{align}
In order to use Divergence Theorem, define the vector field
\[H(y) = \left(\kappa\left(\frac{x-y}{\theta}\right),0,\ldots,0\right).\]
Note that
\begin{align}
    &\diver H (y) =  \frac{-1}{\theta}\frac{\partial}{\partial y_1}\kappa\left(\frac{x-y}{\theta}\right)\\
    &\langle\nabla\mu_{t_j}(y), H(y)\rangle = \kappa\left(\frac{x-y}{\theta}\right) \cdot \frac{\partial}{\partial y_1}\mu_{t_j}(y)\\
    &\diver (\mu_{t_j} H) (y) =  \kappa\left(\frac{x-y}{\theta}\right) \cdot \frac{\partial}{\partial y_1}\mu_{t_j}(y) + \frac{\partial}{\partial y_1}\kappa\left(\frac{x-y}{\theta}\right)\mu_{t_j}(y)
    \end{align}
Using Divergence Theorem, the boundary term in the integral vanishes:

\begin{align}
    \int_{\Pi^d} \diver (\mu_{t_j} H) (y)dy &= \int_{\partial\Pi^d} \langle(\mu_{t_j} H) (y), \hat{n}\rangle dy\\
    &= 0,
    \end{align}
where $\hat{n}$ is the outward unit normal vector. Then, 
\begin{align}
    0 = \int \diver (\mu_{t_j} H) (y)dy = \int \mu_{t_j}(\diver H (y))dy + \int \langle\nabla \mu_{t_j} (y), H(y)\rangle dy
\end{align}
which gives 
\begin{align}
    \frac{-1}{\theta}\int \mu_{t_j}(y) \frac{\partial}{\partial y_1}\kappa\left(\frac{x-y}{\theta}\right)dy = - \int\kappa\left(\frac{x-y}{\theta}\right) \cdot \frac{\partial}{\partial y_1}\mu_{t_j}(y)dy.
\end{align}

Repeating the same argument for each coordinate,

\begin{align}
    \E\widehat{\nabla \mu}_{t_j}(x)
    &=\frac{1}{\theta^{d+1}}\left(\theta\int\kappa\left(\frac{x-y}{\theta}\right)\frac{\partial}{\partial y_1}\mu_{t_j}(y)dy, \ldots, \theta\int\kappa\left(\frac{x-y}{\theta}\right) \frac{\partial}{\partial y_d}\mu_{t_j}(y)dy\right)\\
    &=\frac{1}{\theta^{d}}\int\kappa\left(\frac{x-y}{\theta}\right)\nabla\mu_{t_j}(y)dy.
\end{align}

Therefore, changing variables 
\begin{align*}
    \E\widehat{\nabla \mu}_{t_j}(x)-\nabla \mu_{t_j}(x)
    &\leq\int_{\Pi^d}\kappa\left(u\right)(\nabla\mu_{t_j}(x-\theta u)-\nabla \mu_{t_j}(x))du\\
\end{align*}
and Mean Value Inequality
\begin{align*}
\|\E\widehat{\nabla \mu}_{t_j}(x)-\nabla \mu_{t_j}(x)\|^2
&\leq R^2\theta^2\|\nabla^2\mu_{t_j}\|_\infty^2.
\end{align*}
Finally, the variance is bounded by using \eqref{boundvar}

\begin{align}
\V\widehat{\nabla \mu}_{t_j}(x)=\frac{1}{N\theta^{2d+2}}\V\kappa\left(\frac{x-X_{t_j}^i}{\theta}\right)\leq\frac{1}{N\theta^{2d+2}}\E\kappa^2\left(\frac{x-X_{t_j}^i}{\theta}\right)\leq\frac{\|\mu_{t_j}\|_\infty}{N\theta^{d+2}}\eta
\end{align}
therefore the MSE satisfies
\begin{align}
    \E\|\widehat{\nabla \mu}_{t_j}(x)-\nabla \mu_{t_j}(x)\|^2\leq R^2\theta^2\|\nabla^2\mu_{t_j}\|_\infty^2+\frac{\|\mu_{t_j}\|_\infty}{N\theta^{d+2}}\eta.
\end{align}
\end{proof}

Finally, we can construct the MSE bound for the derivative of log-density estimator.

\subsection{Kernel derivative log density estimator}

\begin{lema}(MSE derivative log density estimator)\label{lem:msederlogdens}
Suppose for all $y\in\Pi^d$
\begin{equation}\label{kernhyp1}
    \|\nabla \kappa(y)\|\leq C_1|\kappa(y)|
\end{equation}
and Assumptions 1 and 2 holds.
Then the estimator defined in \eqref{logdenhat} satisfies
\begin{align*}
\E\|\widehat{\left(\frac{\nabla \mu_{t_j}(x)}{\mu_{t_j}(x)}\right)}-\frac{\nabla \mu_{t_j}(x)}{\mu_{t_j}(x)}\|^2
&\leq\frac{C_3^2}{\theta^2}\exp\left(-N C_{\theta,R}\right)+\frac{8}{\mu_{t_j}^2(x)}\left(R^2\theta^2\|\nabla^2\mu_{t_j}\|_\infty^2+\frac{\|\mu_{t_j}\|_\infty}{N\theta^{d+2}}\eta.\right)\\
&+\frac{8\|\nabla\mu_{t_j}(x)\|^2}{\mu_{t_j}^2(x)}\left(\frac{C_2^2\theta^4}{4}\|\nabla^2\mu_{t_j}\|_\infty^2+\frac{\|\mu_{t_j}\|_\infty}{N\theta^d}\eta\right).
\end{align*}
\end{lema}

\begin{proof}
The idea is to find this upper bound in terms of the MSE of $\hat{\mu}_t(x)$ and $\widehat{\nabla \mu}_t(x)$.

Defining the set
\[A = \left\{\frac{1}{N\theta^d}\sum_{i=1}^N\kappa\left(\frac{x-X_{t_j}^i}{\theta}\right)\geq \frac{\mu_{t_j}(x)}{2}\right\}=\left\{\hat{\mu}_{t_j}(x)\geq\frac{\mu_{t_j}(x)}{2}\right\}\]

We have

\begin{align*}
    \E\left\|\widehat{\left(\frac{\nabla \mu_{t_j}(x)}{\mu_{t_j}(x)}\right)}-\frac{\nabla \mu_{t_j}(x)}{\mu_{t_j}(x)}\right\|^2=\E\bbone_A\left\|\widehat{\left(\frac{\nabla \mu_{t_j}(x)}{\mu_{t_j}(x)}\right)}-\frac{\nabla \mu_{t_j}(x)}{\mu_{t_j}(x)}\right\|^2+\E\bbone_{A^c}\left\|\widehat{\left(\frac{\nabla \mu_{t_j}(x)}{\mu_{t_j}(x)}\right)}-\frac{\nabla \mu_{t_j}(x)}{\mu_{t_j}(x)}\right\|^2.
\end{align*}

If $Z_\theta(t_j,x)\neq0,$ we can use \eqref{kernhyp1}

\begin{align*}
\left\|\frac{1}{\theta}\frac{\sum_{i=1}^N\nabla\kappa\left(\frac{x-X_{t_j}^i}{\theta}\right)}{\sum_{i=1}^N\kappa\left(\frac{x-X_{t_j}^i}{\theta}\right)}-\frac{\nabla\mu_{t_j}(x)}{\mu_{t_j}(x)}\right\|
&\leq\frac{1}{\theta}\frac{1}{\left|\sum_{i=1}^N\kappa\left(\frac{x-X_{t_j}^i}{\theta}\right)\right|}\sum_{i=1}^N\left\|\nabla\kappa\left(\frac{x-X_{t_j}^i}{\theta}\right)\right\|+L_t\\
&\leq\frac{C_1}{\theta}+L_t\coloneqq \frac{C}{\theta}.
\end{align*}

Assumption 2 is to apply Theorem \ref{teo:c0good}:

\begin{align}\label{cota}
    L_t\geq\|\nabla\log \mu_t(x)\| =\left\| \frac{\nabla \mu_t(x)}{\mu_t(x)}\right\|.
\end{align}
Therefore, when $Z_\theta(t_j,x)=0$, we can use \eqref{cota} giving 
 
 \begin{align*}
   \left\|\frac{1}{\theta}\frac{\sum_{i=1}^N\nabla\kappa\left(\frac{x-X_{t_j}^i}{\theta}\right)}{\sum_{i=1}^N\kappa\left(\frac{x-X_{t_j}^i}{\theta}\right)}-\frac{\nabla\mu_{t_j}(x)}{\mu_{t_j}(x)}\right\|&=\left\|\frac{\nabla\mu_{t_j}(x)}{\mu_{t_j}(x)}\right\|\leq L_t.
\end{align*}
In either cases, 
\begin{align}\label{cota1}
\left\|\frac{1}{\theta}\frac{\sum_{i=1}^N\nabla\kappa\left(\frac{x-X_{t_j}^i}{\theta}\right)}{\sum_{i=1}^N\kappa\left(\frac{x-X_{t_j}^i}{\theta}\right)}-\frac{\nabla\mu_{t_j}(x)}{\mu_{t_j}(x)}\right\|\leq\max\{C/\theta, L_t\} \eqqcolon \frac{C_3}{\theta} \quad\forall x.    
\end{align}

With \eqref{cota1} and Lemma \ref{lem:smallprob},

\begin{align}\label{conta}
    \E\left\|\widehat{\left(\frac{\nabla \mu_{t_j}(x)}{\mu_{t_j}(x)}\right)}-\frac{\nabla \mu_{t_j}(x)}{\mu_{t_j}(x)}\right\|^2
    &\leq\frac{C_3^2}{\theta^2}\exp\left(-N C_{\theta,R}\right)+\E\bbone_A\left\|\widehat{\left(\frac{\nabla \mu_{t_j}(x)}{\mu_{t_j}(x)}\right)}-\frac{\nabla \mu_{t_j}(x)}{\mu_{t_j}(x)}\right\|^2.
\end{align}

Now let us concentrate in finding a bound on \eqref{conta}:

\begin{align*}
    \E\bbone_A\left\|\widehat{\left(\frac{\nabla \mu_{t_j}(x)}{\mu_{t_j}(x)}\right)}-\frac{\nabla \mu_{t_j}(x)}{\mu_{t_j}(x)}\right\|^2
    &=\E\bbone_A\left\|\frac{\widehat{\nabla \mu}_{t_j}(x)-\nabla \mu_{t_j}(x)}{\hat{\mu}_{t_j}(x)}+\nabla\mu_{t_j}(x)\left(\frac{\mu_{t_j}(x)-\hat{\mu}_{t_j}(x)}{\hat{\mu}_{t_j}(x)\mu_{t_j}(x)}\right)\right\|^2.
\end{align*}

Bounding each term separately:
\begin{align}
    \E\bbone_A\left\|\frac{\widehat{\nabla \mu}_{t_j}(x)-\nabla \mu_{t_j}(x)}{\hat{\mu}_{t_j}(x)}\right\|^2
    &\leq\frac{4}{\mu_{t_j}^2(x)}\E\|\widehat{\nabla \mu}_{t_j}(x)-\nabla \mu_{t_j}(x)\|^2
\end{align}

and
\begin{align}
  \E\bbone_A\left\|\nabla\mu_{t_j}(x)\left(\frac{\mu_{t_j}(x)-\hat{\mu}_{t_j}(x)}{\hat{\mu}_{t_j}(x)\mu_{t_j}(x)}\right)\right\|^2&=\E\bbone_A\frac{1}{\hat{\mu}_{t_j}^2(x)}\left\|\nabla\mu_{t_j}(x)\left(\frac{\mu_{t_j}(x)-\hat{\mu}_{t_j}(x)}{\mu_{t_j}(x)}\right)\right\|^2\\
  &\leq\frac{4\|\nabla\mu_{t_j}(x)\|^2}{\mu_{t_j}^2(x)}\E\|\mu_{t_j}(x)-\hat{\mu}_{t_j}(x)\|^2.
\end{align}

We can conclude the upper bound on the MSE of $\widehat{\left(\frac{\nabla \mu_{t_j}(x)}{\mu_{t_j}(x)}\right)}$ in terms of the MSE of $\hat{\mu}_{t_j}(x)$ and $\widehat{\nabla \mu}_{t_j}(x)$ with Lemma \ref{lem:derdens} and Lemma \ref{lem:dens}:

\begin{align*}
&\E\left\|\widehat{\left(\frac{\nabla \mu_{t_j}(x)}{\mu_{t_j}(x)}\right)}-\frac{\nabla \mu_{t_j}(x)}{\mu_{t_j}(x)}\right\|^2
\leq \frac{C_3^2}{\theta^2}\exp\left(-N C_{\theta,R}\right)+2\cdot\frac{4}{\mu_{t_j}^2(x)}\left(R^2\theta^2\|\nabla^2\mu_{t_j}\|_\infty^2+\frac{\|\mu_{t_j}\|_\infty}{N\theta^{d+2}}\eta\right)+\\
&+2\cdot\frac{4\|\nabla\mu_{t_j}(x)\|^2}{\mu_{t_j}^2(x)}\left(\frac{C_2^2\theta^4}{4}\|\nabla^2\mu_{t_j}\|_\infty^2+\frac{\|\mu_{t_j}\|_\infty}{N\theta^d}\eta\right).
\end{align*}

\end{proof}

\subsection{Potential estimator}

As usual, define for $j=0,\ldots, M$:
\[\widehat{\nabla V_{t_j}(x)}\coloneqq\begin{cases}
	        0, \text{ if }  Z_\theta(t_j,x)=0,\\
	        -\widehat{v}_{t_j}(x)-\frac{1}{\theta}\frac{\sum_{i=1}^N\nabla\kappa\left(\frac{x-X_{t_j}^i}{\theta}\right)}{\sum_{i=1}^N\kappa\left(\frac{x-X_{t_j}^i}{\theta}\right)}, \text{ otherwise.}
	\end{cases}\]
For $t \in(t_j, t_{j+1})$ for some $j = 0, \ldots, M$, we are going to linearly interpolate it between time-steps as we did in Remark \ref{rem:interpolation}.

Finally, the main goal of this section:

\begin{teo}(MSE POTENTIAL)\label{teo:msedrift}
    Assume that Assumptions 1, 2 and 3 holds and that the kernel $\kappa$ satisfies $\supp(\kappa) \subset [-R/2,R/2]^d$, \eqref{kernhyp1} and \eqref{kernhyp2}:
\begin{equation}\label{kernhyp2}
   \int_{\Pi^d} u\kappa(u)du=0.
\end{equation}
    In the hypothesis of Theorem \ref{teo:c0good}, we have then
\begin{align*}
    \E\|\widehat{\nabla V_{t_j}}(x)-\nabla V_{t_j}(x)\|^2
    &\leq C\exp\left(-N C_{\theta,R}\right)+32(LhC)^2+2(R\theta L)^2+ \frac{16\sigma^2C}{N\theta^dR^{2d}}\\
    &+\frac{C}{\theta^2}\exp\left(-N C_{\theta,R}\right)+CR^2\theta^2+C\theta^4+\frac{C}{N}\left(\frac{1}{\theta^d}+\frac{1}{\theta^{d+2}}\right).
\end{align*}
\end{teo}
\begin{proof}
Using Theorem \ref{teo:msevhat} and Lemma \ref{lem:msederlogdens}, 
\begin{align*}
    &\E\|\widehat{\nabla V_{t_j}}(x)-\nabla V_{t_j}(x)\|^2
    \leq2\E\|\widehat{v}_{t_j}(x)-v_{t_j}(x)\|^2+2\E\|\widehat{\nabla \log\mu_{t_j}(x)}-\nabla \log\mu_{t_j}(x)\|^2\\
&\leq 2\left(C\exp\left(-N C_{\theta,R}\right)+16(LhC)^2+(R\theta L)^2+ \frac{8\sigma^2C}{N\theta^dR^{2d}}
    \right) \\
    &+\frac{C\exp\left(-N C_{\theta,R}\right)}{\theta^2}+\frac{16}{\mu_{t_j}^2(x)}\left(R^2\theta^2\|\nabla^2\mu_{t_j}\|_\infty^2+\frac{\|\mu_{t_j}\|_\infty}{N\theta^{d+2}}\eta\right)\\
&+\frac{16\|\nabla\mu_{t_j}(x)\|^2}{\mu_{t_j}^2(x)}\left(\frac{C\theta^4}{4}\|\nabla^2\mu_{t_j}\|_\infty^2+\frac{\|\mu_{t_j}\|_\infty}{N\theta^d}\eta\right).
\end{align*}

We can use Theorem \ref{teo:c0good} and Corollary \ref{cor:bounddens} to bound the norms $\|\nabla^2\mu_{t_j}\|_\infty,\|\nabla\mu_{t_j}(x)\|$ and $\|\mu_{t_j}\|_\infty$ uniformly in time.

Therefore, 
\begin{align*}
    \E\|\widehat{\nabla V_{t_j}}(x)-\nabla V_{t_j}(x)\|^2
    &\leq C\exp\left(-N C_{\theta,R}\right)+32(LhC)^2+2(R\theta L)^2+ \frac{16\sigma^2C}{N\theta^dR^{2d}}\\
    &+\frac{C}{\theta^2}\exp\left(-N C_{\theta,R}\right)+CR^2\theta^2+C\theta^4+\frac{C}{N}\left(\frac{1}{\theta^d}+\frac{1}{\theta^{d+2}}\right).
\end{align*}
\end{proof}

\subsubsection{Time-invariance of drift}

Now suppose that the drift term does not depend on time, i.e, choose any $t^*\in\{t_0, t_1, \ldots, t_M\}$ and consider
\[\nabla V(x)=-v_{t^*}(x)-\nabla\log\mu_{t^*}(x) \]
Since we are dividing the interval $[0,T]$ in $M$ parts of size $h$, i.e. $M\cdot h=T,$ where $h=t_{j+1}-t_j$, we can take the sum 
\begin{align}\label{soma}
    \nabla V(x)=\frac{1}{M+1}\sum_{j=0}^M(-v_{t_j}(x)-\nabla\log\mu_{t_j}(x)).
\end{align}
In that case we can rewrite the estimator $\widehat{\nabla V}(x)$ using \eqref{soma}:
\begin{align}
\widehat{\nabla V}(x)=\frac{1}{M+1}\sum_{j=0}^M(-\widehat{v}_{t_j}(x)-\widehat{\nabla\log\mu}_{t_j}).
\end{align}

\begin{teo}(MSE HOMOGENEOUS POTENTIAL)\label{teo:msedriftsemt}
In the hypotheses of Theorem \ref{teo:msedrift},
\begin{align*}
\E\|\widehat{\nabla V}(x)-\nabla V(x)\|^2&\leq C\exp\left(-N C_{\theta,R}\right)\\
&+\frac{d\sigma^2C}{MN\theta^d(R\theta)^{2d}}+C(Lh)^2+C(R\theta L)^2\\
&+\frac{C}{\theta^2}\exp\left(-N C_{\theta,R}\right)+CR^2\theta^2+C\theta^4+\frac{C}{N}\left(\frac{1}{\theta^d}+\frac{1}{\theta^{d+2}}\right).
\end{align*}

\end{teo}

\begin{proof}

For fixed $x\in\Pi^d,$ we will follow the same ideas as in Theorem \ref{teo:msevhat}. 
By triangle inequality,

\begin{align*}
    \E\|\widehat{\nabla V}(x)-\nabla V(x)\|^2&\leq2\E\bbone_E\left\|\frac{1}{M}\sum_{j=1}^M\widehat{v}_{t_j}(x)-v_{t_j}(x)\right\|^2+2\E\bbone_E\left\|\frac{1}{M}\sum_{j=1}^M\widehat{\nabla\log\mu}_{t_j}-\nabla\log\mu_{t_j}(x)\right\|^2\\
    &+2\E\bbone_{E^c}\left\|\frac{1}{M}\sum_{j=1}^M\widehat{v}_{t_j}(x)-v_{t_j}(x)\right\|^2+2\E\bbone_{E^c}\left\|\frac{1}{M}\sum_{j=1}^M\widehat{\nabla\log\mu}_{t_j}-\nabla\log\mu_{t_j}(x)\right\|^2.
\end{align*}

We can apply Lemma \ref{lem:smallprob} to obtain a bound on the set $E^C$:
\begin{align}
    2\E\bbone_{E^c}\left\|\frac{1}{M}\sum_{j=1}^M\widehat{v}_{t_j}(x)-v_{t_j}(x)\right\|^2
    &\leq2\mathbb{P}(E^c)(2+\sup_{t,x}\|v_{t_j}(x)\|^2)\\
    &\leq C\exp\left(-N C_{\theta,R}\right)
\end{align}

and 
\begin{align}
    2\E\bbone_{E^c}\left\|\frac{1}{M}\sum_{j=1}^M\widehat{\nabla\log\mu}_{t_j}-\nabla\log\mu_{t_j}(x)\right\|^2
    &\leq\frac{2\mathbb{P}(E^c)}{C_1^2}\|\nabla\mu_{t_j}\|_\infty^2\\
    &\leq C\exp\left(-N C_{\theta,R}\right).
\end{align}

The analysis on $E$ is quite similar to the one in Theorem \ref{teo:msevhat}, but now we have a sum on $j:$

\begin{align*}
    &\left\|\frac{1}{M}\sum_{j=1}^M\widehat{v}_{t_j}(x)-v_{t_j}(x)\right\|^2=\\
    &\leq 2\left\|\frac{1}{M}\sum_{j=1}^M\frac{1}{Z_\theta^j(x)\theta^d}\left(\sum_{i=1}^N\kappa\left(\frac{x-X_{t_j}^i}{\theta}\right)Y_{j,i}-v_{t_j}(x)\right)\right\|^2 + 2\left\|\frac{1}{M}\sum_{j=1}^M\frac{1}{Z_\theta^j(x)\theta^d}\sum_{i=1}^N\kappa\left(\frac{x-X_{t_j}^i}{\theta}\right)\epsilon_{j,i}\right\|^2.
\end{align*}

The first term is completely analogous:
\begin{align}
    \E\bbone_{E}\left\|\frac{1}{M}\sum_{j=1}^M\frac{1}{Z_\theta^j(x)\theta^d}\sum_{i=1}^N\kappa\left(\frac{x-X_{t_j}^i}{\theta}\right)Y_{j,i}-v_{t_j}(x)\right\|^2
    \leq2L^2h^2C^2+2L^2R^2\theta^2.
\end{align}

In order to bound the second one, we will take conditional expectation with respect to $\mathcal{F}^N$ and use the inequalities in Lemma \ref{lem:boundkappa}:

\begin{align*}
    &\E \bbone_E\left\|\frac{1}{M}\sum_{j=1}^M\frac{1}{Z_\theta^j(x)\theta^d}\sum_{i=1}^N\kappa\left(\frac{x-X_{t_j}^i}{\theta}\right)\epsilon_{j,i}\right\|^2 = \frac{1}{M^2}\E \bbone_E\left\|\sum_{j=1}^M\frac{1}{Z_\theta^j(x)\theta^d}\sum_{i=1}^N\kappa\left(\frac{x-X_{t_j}^i}{\theta}\right)\epsilon_{j,i}\right\|^2\\
    &=\frac{1}{M^2}\E\left[\E\left[\sum_{j=1}^M\frac{1}{Z_\theta^j(x)\theta^d}\sum_{i=1}^N\kappa\left(\frac{x-X_{t_j}^i}{\theta}\right)\epsilon_{j,i}\vert\mathcal{F}^N\right]\right]\\
    &=\frac{1}{M^2}\E\left[\bbone_E\sum_{j=1}^M\frac{1}{Z_\theta^{2j}(x)\theta^{2d}}\sum_{i=1}^N\kappa^2\left(\frac{x-X_{t_j}^i}{\theta}\right)\tr\cov \epsilon_{j,i}\right]\\
    &=\frac{d\sigma^2}{M^2\theta^{2d}}\E\left[\bbone_E\sum_{j=1}^M\frac{1}{Z_\theta^{2j}(x)}\sum_{i=1}^N\kappa^2\left(\frac{x-X_{t_j}^i}{\theta}\right)\right]\\
    &\leq\frac{d\sigma^2\lambda_2\|\kappa\|_\infty2^{4d}}{M^2\theta^{2d}\lambda_1^2C_4^2R^{2d}} \frac{4MN}{CN^2} = \frac{d\sigma^2C}{MN\theta^d(R\theta)^{2d}}.
\end{align*}

And to bound the last term on $E$ term we can use convexity of $\|.\|^2$ and the same approach as Theorem \ref{teo:msedrift}:

\begin{multline*}
    2\E\bbone_E\left\|\frac{1}{M}\sum_{j=1}^M\widehat{\nabla\log\mu}_{t_j}-\nabla\log\mu_{t_j}(x)\right\|^2\leq
    \frac{2}{M}\sum_{j=1}^M\E\left\|\widehat{\nabla\log\mu}_{t_j}-\nabla\log\mu_{t_j}(x)\right\|^2\\
    \leq\frac{2}{M}\sum_{j=1}^M
        \Biggl\{\frac{C}{\theta^2}\exp\left(-N C_{\theta,R}\right)+CR^2\theta^2+C\theta^4+\frac{C}{N}\left(\frac{1}{\theta^d}+\frac{1}{\theta^{d+2}}\right)\Biggr\}.
\end{multline*}

\end{proof}

\section{Kernel-interaction estimator: McKean-Vlasov case}\label{sec:mve}

In this section we are going to consider the McKean-Vlasov equation. Now, $v_{t} = F \ast \mu_{t} - \nabla \log \mu_t$ where $F:\Pi^d\to\R^d$ and the goal is to estimate $F$ using deconvolution techniques provided this quotient is properly defined:
\begin{equation}\label{conv}
  \mathcal{F}[v_{t}] = \mathcal{F}[F]\cdot\mathcal{F}[\mu_{t}] \Longrightarrow \mathcal{F}[\widehat{F}]=\frac{\mathcal{F}[\hat{v}_t + \widehat{\nabla \log \mu_t}]}{\hat{\mathcal{F}}[\mu_{t}]}.  
\end{equation}
This estimator is an adaptation to the torus of \citep{Johannes_2009} estimator. It is important to say that although the Fourier transform on the torus gives a sequence, nor a function, the same useful properties of Fourier transforms for functions in $\R^d$ still holds, as one can check in \citep{fouanal}.

Note that $F$ is constant in the time variable, which means that we can choose $t^*\in{t_{j_0}, t_{j_1}, \ldots, t_{j_M}}$ and define for $u\in\R^d,$
\begin{equation}\label{chapeumu}
  \hat{\mathcal{F}}[\mu_{t^*}](u)\coloneqq\frac{1}{N}\sum_{k=1}^N\exp(-iu^\top X_{kt^*}).  
\end{equation}

This approach seeks an upper bound on the Fourier transform of $F$, which can be reduced to a bound on the Fourier transform of its coordinate functions. Indeed, since 
\begin{equation}
    \|F\|_{L^2(\Pi^d,\R^d)}=\sum_{i=1}^d\|F^i\|_{L^2(\Pi^d,\R)},
\end{equation}
it is sufficient to obtain a bound on $F^i.$

As mentioned in \citep{Johannes_2009}, the inverse operation of a convolution is not continuous. Then, in order to make this problem well posed, we are going to regularize the denominator in \eqref{conv}. To do so, let us consider the well-known Sobolev space $H_s$ defined by:
\begin{equation}
    H_s\coloneqq\left\{f\in L^2(\Pi^d)\mid \|f\|_s^2\coloneqq\sum_{u\in\Z^d}\ell_s^2(u)|\mathcal{F}[f](u)|^2<\infty\right\},
\end{equation}
where $\ell_s(u)\coloneqq(1+\|u\|^2)^\frac{s}{2}$, for $s\geq0.$ When $f\in L^2(\Pi^d)$, the Fourier transform is an isometry: $\|\mathcal{F}[f]\|^2_{\ell^2}=\|f\|^2_{L^2}$.
Note that $\|\cdot\|_s^2\geq\|\cdot\|^2_{L^2(\Pi^d)}$, since $\ell_s^2(u)\geq1$ for all $u\in\R^d.$ Therefore, upper bounding $\E\|\hat{F^i}-F^i\|^2_s$ means upper bounding $\E\|\hat{F^i}-F^i\|^2_{L^2(\Pi^d)}$.

Take $\alpha>0$ and consider the set $A_s = \left\{u\in\Pi^d\mid\left|\frac{\hat{\F}[\mu_{t^*}](u)}{\ell_s(u)}\right|^2\geq\alpha\right\}$. We are going to adopt the notation $\hat{g}_t = \hat{v}_{t^*}+\widehat{\nabla \log \mu_{t^*}}$ and define the estimator of $F^i$ through its regularized Fourier transform:

\begin{equation}\label{fouhat}
    \F[\widehat{F^i}]_s\coloneqq \frac{\F[\hat{g}_{t^*}^i] \cdot\overline{\hat{\F}[\mu_{t^*}]}}{|\hat{\F}[\mu_{t^*}]|^2}\cdot\bbone_{A_s}.
\end{equation}

The regularized version of $\F[F^i]$ is the following:
\begin{eqnarray}
    \F[{F^{i\alpha}}]\coloneqq& \F[F^i]\cdot\bbone_{A_s}\\
    =& \frac{\F[g_{t^*}^i] \cdot\overline{\F[\mu_{t^*}]}}{|\F[\mu_{t^*}]|^2}\cdot\bbone_{A_s}.
\end{eqnarray}

Our previous results give us that $\E\|\hat{g}_{t^*}^i-g_{t^*}^i\|_{L^2}=o(1)$ as $N\to\infty$ and $h\to0.$

\begin{prop} In the same hypothesis of Theorem \ref{teo:msevhat} and Lemma \ref{lem:msederlogdens}, for any $i\in\{1, 2, \ldots d\}$ and any $t\in[0,T]$,
    \[\E\|\hat{v}_t^i-v_t^i\|^2_{L^2}=o(1) \text{ as } N\to\infty \text{ and } h\to0.\]
and 
\[\E\|\widehat{\nabla \log \mu_{t^*}}^i-\nabla \log \mu_{t^*}^i\|^2_{L^2}=o(1) \text{ as } N\to\infty \text{ and } h\to0.\]
\end{prop}
\begin{proof} Theorem \ref{teo:msevhat} gives an uniform bound on MSE of $\hat{v}$. Since our domain is bounded, we just have to apply Fubini for a fixed $i_0\in\{1, 2, \ldots d\}:$
\begin{align*}
    \E\|\hat{v}_t^{i_0}-v_t^{i_0}\|^2_{L^2(\Pi^d)}\leq\sum_{i=1}^d\E\|\hat{v}^i-v^i\|^2_{L^2(\Pi^d)}=\E\|\hat{v}_t-v_t\|^2_{L^2(\Pi^d, \R^d)}&=\E\int_{\Pi^d}|\hat{v}_t(x)-v_t(x)|^2dx\\
    &=\int_{\Pi^d}\E|\hat{v}_t(x)-v_t(x)|^2dx \\
    &\leq C \cdot o(1).
\end{align*}
The same reasoning applies for $\widehat{\nabla \log \mu_{t^*}}^i$ using Lemma \ref{lem:msederlogdens}. 
\end{proof}

{\color{red} Since our deconvolution technique is based on the quotient \eqref{conv}, we need to impose an additional assumption on the Fourier transform of the solution.
\paragraph{Assumption 4 } For all $u\in\mathbb{Z}^d$, $|\F[\mu_{t^*}](u)|^2>0$.

Whether the Fourier transform of $\mu_t$ vanishes or not is a delicate analytical question. As our focus lies on constructing an inference technique for the interaction kernel, we do not address the specific conditions under which Assumption 4 holds. As one can check in \citep{laetita} and \citep{bayinf} that kind of assumption is typical in this context. For the interested reader, discussion about analytic properties of the Fourier transform of $\mu_t$ can be found in \citep{amorino2024polynomial}.}

Now we are ready for the main result of this section. The results here are multidimensional analogues on the torus of the results in \citep{Johannes_2009}. Denote the $i-$th coordinate function of $\hat{v}_{t^*}+ \widehat{\nabla \log \mu_{t^*}}$ by $\hat{g}_{t^*}^i$, for $i \in\{1, \ldots, d\}$.

\begin{teo}
    Assume the hypothesis of Theorem \ref{teo:msevhat}, Lemma \ref{lem:msederlogdens} {\color{red} and Assumption 4.} Suppose $F^i\in H_p$, for $p\geq0$. Consider $\F[\widehat{F^i}]_s$ defined in \eqref{fouhat} for $0\leq s\leq p$ with $\alpha = o(1)$ as $N\to\infty$ such that $(\frac{1}{\alpha N}\vee\frac{\E\|\hat{g}_{t^*}^i-g_{t^*}^i\|_{L^2}^2}{\alpha}) = o(1)$ as $N\to\infty$ and $h\to0$. Then
    \begin{equation}
        \E\|\hat{F}^i-F^i\|^2_{L^2(\Pi^d)} = o(1) \text{ as } N\to\infty \text{ and } h\to0.
    \end{equation}
\end{teo}
\begin{proof}
    As already mentioned, it is sufficient to get a bound on $\E\|\hat{F^i}-F^i\|^2_s$.
Notice that
\begin{eqnarray}
    \E\|\hat{F^i}-F^i\|^2_s\leq 2 \E\|\hat{F^i}-F^{i\alpha}\|^2_s+2 \E\|F^{i\alpha}-F^i\|^2_s.
\end{eqnarray}

We will show that
\begin{align}
    &\E\|\hat{F}^i-F^{i\alpha}\|^2_s\leq \frac{2}{\alpha}\E\|\hat{g}_{t^*}^i-g_{t^*}^i\|^2_{L^2(\Pi^d)}+\frac{2C}{\alpha N}\|F^i\|_s^2,\label{a12}\\
    &\E\|F^{i\alpha}-F^i\|^2_s = o(1)\text{ as } N\to\infty.\label{a13}
\end{align}

\textbf{Proof of \eqref{a12}:}
\begin{multline*}
    \E\|\hat{F}^i-F^{i\alpha}\|^2_s
    \leq 2 \E\sum_{u\in\Z^d}\frac{\ell_s^2(u)}{|\hat{\F}[\mu_{t^*}](u)|^2}\frac{|\overline{\hat{\F}[\mu_{t^*}]}(u)|^2}{|\hat{\F}[\mu_{t^*}](u)|^2}\left|\F[\hat{g}_{t^*}^i](u)-\F[g_{t^*}^i](u)\right|^2\cdot\bbone_{A_s}(u)\\
    +2\E\sum_{u\in\Z^d}\ell_s^2(u)\left|\frac{\F[g_{t^*}^i](u)\cdot\overline{\hat{\F}[\mu_{t^*}](u)}}{|\hat{\F}[\mu_{t^*}](u)|^2}\cdot\bbone_{A_s}(u)-\frac{\F[g_{t^*}^i](u) \cdot\overline{\F[\mu_{t^*}](u)}}{|\F[\mu_{t^*}](u)|^2}\cdot\bbone_{A_s}(u)\right|^2.
\end{multline*}
The first term is bounded by $\frac{2}{\alpha}\E\|\hat{g}^i-g^i\|_{L^2}^2$ and the second one can be written as follows:
\begin{align*}
&\E\sum_{u\in\Z^d}\ell_s^2(u)\left|\F[g_{t^*}^i](u)\bbone_{A_s}(u)\left(\frac{\overline{\hat{\F}[\mu_{t^*}](u)}}{|\hat{\F}[\mu_{t^*}](u)|^2}-\frac{\overline{\F[\mu_{t^*}](u)}}{|\F[\mu_{t^*}](u)|^2}\right)\right|^2\\
&=\E\sum_{u\in\Z^d}\ell_s^2(u)\left|\F[g_{t^*}^i](u)\frac{\overline{\F[\mu_{t^*}](u)}}{|\F[\mu_{t^*}](u)|^2}\bbone_{A_s}(u)\left(\frac{|\F[\mu_{t^*}](u)|^2\overline{\hat{\F}[\mu_{t^*}](u)}}{\overline{\F[\mu_{t^*}](u)}|\hat{\F}[\mu_{t^*}](u)|^2}-\frac{|\F[\mu_{t^*}](u)|^2\overline{\F[\mu_{t^*}](u)}}{\overline{\F[\mu_{t^*}](u)}|\F[\mu_{t^*}](u)|^2}\right)\right|^2\\
&=\E\sum_{u\in\Z^d}\ell_s^2(u)\left|\F[F^i](u)\bbone_{A_s}(u)\left(\frac{\F[\mu_{t^*}](u)}{\hat{\F}[\mu_{t^*}](u)}-1\right)\right|^2.
\end{align*}

Note that we can change the order between $\E$ and the sum $\sum_{u\in\Z^d}$ since in the space $H_s$, the series is convergent. By Lemma \ref{A.2}, 
\begin{align*}
\|\ell_s\cdot\F[F^i]\cdot\left(\E\left[\bbone_{A_s}\left|\frac{\F[\mu_{t^*}]-\hat{\F}[\mu_{t^*}]}{\hat{\F}[\mu_{t^*}]}\right|^2\right]\right)^{\frac{1}{2}}\|_{\ell^2}^2&\leq\|\ell_s\cdot\F[F^i]\cdot\left(\frac{C}{\alpha N}\right)^\frac{1}{2}\|_{\ell^2}^2\\
&=\|F^i\|_s^2\cdot\frac{C}{\alpha N}\\
&\leq\|F^i\|_p^2\cdot\frac{C}{\alpha N}.
\end{align*}

Since $F^i\in H_p$, we can conclude that \eqref{a12} goes to zero.

\textbf{Proof of \eqref{a13}:}
In order to use Dominated Convergence Theorem, let us see that this function is dominated:
\begin{align}
\|F^{i\alpha}-F^i\|^2_s 
&=\sum_{u\in\Z^d}\ell_s(u)^2|\F[F^i]\cdot\bbone_{A_s^C}(u)|^2\\
&=\sum_{u\in\Z^d}\ell_s(u)^2|\F[F^i](u)|^2\bbone_{A_s^C}(u)\\
&=\|\ell_s\cdot\F[F^i]\cdot\left(\bbone_{A_s^C}\right)^\frac{1}{2}\|_{\ell^2}^2\\
&\leq\|\ell_s\cdot\F[F^i]\|_{\ell^2}^2=\|F^i\|_s^2\leq\|F^i\|_p^2<\infty.
\end{align}

If we can prove that $\E\bbone_{A_s^C}$ goes to zero with $\alpha\to0$ and $N\to\infty$, we are done. Take $\alpha_0$ such that $\left|\frac{\F[\mu_{t^*}]}{2\ell_s}\right|^2>\alpha$ for $\alpha\leq\alpha_0.$ Using Chebyshev, 

\begin{align}
    \E\bbone_{A_s^C}(u)&=\mathbb{P}\left(\left|\frac{\hat{\F}[\mu_{t^*}](u)}{\ell_s(u)}\right|^2<\alpha\right)\\
    &\leq\mathbb{P}\left(\left|\hat{\F}[\mu_{t^*}](u)-\F[\mu_{t^*}](u)\right|>\frac{|\F[\mu_{t^*}](u)|}{2}\right)\\
    &\leq \frac{4\V(\hat{\F}[\mu_{t^*}](u))}{|\F[\mu_{t^*}](u)|^2}\leq \frac{4\V(\hat{\F}[\mu_{t^*}](u))}{4\ell_s^2(u)\alpha}\\
    &=\frac{\V(\hat{\F}[\mu_{t^*}](u)/\ell_s(u))}{\alpha}.
\end{align}

By Lemma \ref{A.2}, 
\begin{align}
    \frac{\V(\hat{\F}[\mu_{t^*}](u)/\ell_s(u))}{\alpha}\leq \frac{C}{\alpha N}.
\end{align}

\end{proof}

\begin{lema}\label{A.2}
    Consider the estimator $\hat{\F}[\mu_{t^*}]$ defined in \eqref{chapeumu}. Then
    \begin{enumerate}
        \item \begin{align}
        \E\left|\frac{\hat{\F}[\mu_{t^*}](u)}{\ell_s(u)}-\frac{\F[\mu_{t^*}](u)}{\ell_s(u)}\right|^2\leq \frac{C}{N}.
            \end{align}
        \item \begin{align*}
        &\E\left[\bbone\{\left|\frac{\hat{\F}[\mu_{t^*}](u)}{\ell_s(u)}\right|^2\geq\alpha\}\cdot\frac{|\hat{\F}[\mu_{t^*}](u)-\F[\mu_{t^*}](u)|^2}{|\hat{\F}[\mu_{t^*}](u)|^2}\right]\leq\frac{C}{N\alpha}.
    \end{align*}
    \end{enumerate}
\end{lema}
\begin{proof}
    Define $Z_k\coloneqq\exp(-it^\top X_{kt^*})-\F[\mu_{t^*}](u).$ Since $|\exp(-it^\top X_{kt^*})|\leq1$ and $\ell_s(u)\geq1$ for all $t$,
    \begin{align}
        \left|\frac{Z_k}{\ell_s(u)}\right|^2 &= \frac{1}{\ell_s^2(u)}\left|\exp(-it^\top X_{kt^*})-\E\exp(-it^\top X_{kt^*})\right|^2\\
        &\leq\frac{1}{\ell_s^2(u)}2^2\leq 4.
    \end{align}
Applying this bound, we get
\begin{align}\label{A.8}
  \E\left|\frac{\hat{\F}[\mu_{t^*}](u)}{\ell_s(u)}-\frac{\F[\mu_{t^*}](u)}{\ell_s(u)}\right|^{2} \leq\E\sum_{k=1}^N\left|\frac{Z_k}{\ell_s(u)}\right|^2\leq\frac{4}{N}.
\end{align}

For the second item,
\begin{align}
    &\E\left[\bbone\{\left|\frac{\hat{\F}[\mu_{t^*}](u)}{\ell_s(u)}\right|^2\geq\alpha\}\cdot\frac{|\hat{\F}[\mu_{t^*}](u)-\F[\mu_{t^*}](u)|^2}{|\hat{\F}[\mu_{t^*}](u)|^2}\right]=\\
    &=\E\left[\bbone\{\left|\frac{\hat{\F}[\mu_{t^*}](u)}{\ell_s(u)}\right|^2\geq\alpha\}\cdot\frac{|\hat{\F}[\mu_{t^*}](u)/\ell_s(u)-\F[\mu_{t^*}](u)/\ell_s(u)|^2}{|\hat{\F}[\mu_{t^*}](u)/\ell_s(u)|^2}\right]\\
    &\leq\frac{1}{\alpha}\E|\hat{\F}[\mu_{t^*}](u)/\ell_s(u)-\F[\mu_{t^*}](u)/\ell_s(u)|^2.
\end{align}
By \eqref{A.8}, we can conclude:
\begin{align}
    \E\left[\bbone\{\left|\frac{\hat{\F}[\mu_{t^*}](u)}{\ell_s(u)}\right|^2\geq\alpha\}\cdot\frac{|\hat{\F}[\mu_{t^*}](u)-\F[\mu_{t^*}](u)|^2}{|\hat{\F}[\mu_{t^*}](u)|^2}\right]\leq\frac{4}{N\alpha}.
\end{align}
\end{proof}

\section{Regularity estimates of \texorpdfstring{$\mu_t$}{ut}}\label{sec:regest}

In order to obtain bounds on derivatives of $\log\mu_t(x)$ that are used in this chapter, we need the following definition:

\begin{defi}A density $\mu$ is called $C-$good if for any $x\in\Pi^d$
\begin{equation}\label{assump2}
    \left\{\|\nabla\log\mu(x)\|,\|\nabla^2\log\mu(x)\|_{op}\right\}\leq C.
\end{equation}    
\end{defi}

The following result is derived from the proof of Lemma 8 in \citep{artigoFPE}. It guarantees boundedness of the gradient and the Hessian of the log probability of all instants of time.

\begin{teo}\label{teo:c0good}
Suppose that $\mu_0$ is $L_0-$good. Given $T>0$, let the vector field $v_t$ satisfies for any $x\in\Pi^d$ and any $t\in[0,T]$
\begin{equation}\label{assump3}
    \{\|\nabla\diver v_t(x)\|,\|\nabla^2\diver v_t(x)\|_{op}\}\leq L_v 
\end{equation}
and 
\begin{equation}\label{assump4}
    \{\|\nabla v_t(x)\|_{op},\|\nabla^2 v_t(x)\|_{op}\}\leq L_v .
\end{equation}
Then for all $t\in(0,T]$, $\mu_t$ is $L_t-$good, where $L_t$ depends on $L_0$ and $L_v$.
As a consequence, we obtain for all $x\in\Pi^d$
\begin{align}
    \{\|\nabla\mu_t(x)\|, \|\nabla^2\mu_t(x)\|_{op}\}\leq C_t.
\end{align}
\end{teo}

\begin{proof}
To prove that $\nabla\log\mu_t$ is bounded, we can use the integral form of Gronwall's inequality since by Proposition 2 in \citep{artigoFPE},
\[\nabla\log\mu_t(x)=\nabla\log\mu_0(x(0))-\int_0^t \nabla\diver v_s(x(s))ds-\int_0^t \nabla v_s(x(s))^\top \nabla\log\mu_s(x(s))ds.\]
With assumptions \ref{assump2} and \ref{assump3},
\begin{align}
    \|\nabla\log\mu_t(x)\|\leq(L_0+tL_v)\exp(tL_v)\leq(L_0+TL_v)\exp(TL_v)\eqqcolon P_1.
\end{align}
Together with the previous inequality and property \ref{assump1}, we get for free an upper bound of the gradient of the density: 

\begin{align}
P_1\geq\|\nabla\log\mu_t(x)\|=\left\|\frac{\nabla\mu_t(x)}{\mu_t(x)}\right\|\Longrightarrow \|\nabla\mu_t(x)\|\leq \lambda_2\cdot P_1\eqqcolon M_1.
\end{align}
For the second order term, we need Proposition 3 in \citep{artigoFPE}: 
\begin{align*}
    \nabla^2\log\mu_t(x)&=\nabla^2\log\mu_0(x(0))-\int_0^t \nabla^2\diver v_s(x(s))ds-\int_0^t (\nabla^2\log\mu_s(x(s)))^\top  J[v_s(x(s))]ds\\
    &-\int_0^t (J[v_s(x(s))])^\top \nabla^2\log\mu_s(x(s))ds-\int_0^t \nabla^2v_s(x(s))\otimes_1\nabla\log\mu_s(x(s))ds
\end{align*}
so if $B_1=(L_0+TL_v)\exp(TL_v)$
\begin{align*}
    \|\nabla^2\log\mu_t(x)\|_{op}&\leq L_0+tL_v+2\int_0^t L_v\|\nabla^2\log\mu_s(x(s))\|_{op}ds+\int_0^tL_v(L_0+tL_v)\exp(tL_v)ds\\
    &=L_0+tL_v(1+B_1)+2L_v\int_0^1\|\nabla^2\log\mu_s(x(s))\|_{op}ds.
\end{align*}
By Gronwall,
\begin{align*}
\|\nabla^2\log\mu_t(x)\|_{op}\leq (L_0+tL_v(1+B_1))\exp(2tL_v)\leq (L_0+TL_v(1+B_1))\exp(2TL_v)\eqqcolon P_2.
\end{align*}
We can also get for free an upper bound of the Hessian of the density. Indeed, 

\begin{align}
    \|\nabla^2\mu_t(x)\|_{op}&\leq \lambda_2\left\|\frac{\nabla^2\mu_t(x)}{\mu_t(x)}\right\|_{op}\\
    &\leq \lambda_2\left\|\frac{\nabla^2\mu_t(x)}{\mu_t(x)}-\frac{(\nabla\mu_t(x))^\top \nabla\mu_t(x)}{\mu_t^2(x)}\right\|_{op}
    +\lambda_2\|\frac{(\nabla\mu_t(x))^\top \nabla\mu_t(x)}{\mu_t^2(x)}\|_{op}\\
    &= \lambda_2\|\nabla^2\log\mu_t(x)\|_{op} + \lambda_2\left\|\frac{(\nabla\mu_t(x))^\top \nabla\mu_t(x)}{\mu_t^2(x)}\right\|_{op}\\
    &\leq (\lambda_2P_2+\lambda_2\frac{M_1^2}{\lambda_1^2}).
    \end{align}

 \end{proof}

Using fewer hypothesis than the last theorem, we obtain boundedness of $\mu_t$ uniformly in time.
\begin{cor}\label{cor:bounddens}
Suppose that 
\begin{equation}
    \|\nabla \log\mu_0(x)\|\leq L_0\mbox{ and } \max\{\|\nabla\diver v_t(x), \nabla v_t(x)\|\}\leq L_v.
\end{equation}
Then for for all $t\in[0,T]$ and all $x\in\Pi^d$, there existss positive constants $\lambda_1, \lambda_2>0$ such that
    \begin{equation}\label{assump1}
        \lambda_1\leq \mu_t(x)\leq \lambda_2.
    \end{equation}  
\end{cor}
\begin{proof}
    With these assumptions we get boundedness of $\|\nabla \log \mu_t(x)\|$. It implies by Mean Value Inequality that $\log\mu_t(x)$ is L-Lipschitz. Since $\Pi^d$ is bounded, we get for $x,y\in\Pi^d:$
    \[-L\diam\Pi^d\leq \log\frac{\mu_t(x)}{\mu_t(y)}\leq L\diam\Pi^d.\]
    Using $\int_{\Pi^d}\mu_t(y)dy=1,$ we get
    \[\exp(-L\diam\Pi^d)\leq\mu_t(x)Leb(\Pi^d)\leq\exp(L\diam\Pi^d).\]
Note that $\diam\Pi^d= R\sqrt{d}$ and $Leb(\Pi^d)=R^d.$ 
\end{proof}

\section{Technical lemmas}\label{sec:teclem}
We can control the first and second moments of $\kappa_\theta$ since the density $\mu_t$ is bounded from below and from above.
\begin{lema}\label{lem:boundkappa}
    Suppose Assumption 2 holds and let $\kappa$ be such that $\supp(\kappa) \subset [-R/2,R/2]^d$, and $C_4>0$ such that  $\kappa\big|_{[-R/4,R/4]^d}\geq C_4$. 
    
Then

\begin{enumerate}[label=\alph*)]
    \item $\E{\kappa^{i,j}}^2_\theta(x)\leq\frac{\lambda_2\|\kappa\|_\infty}{\theta^{d}}$
    \item $\E\kappa_\theta^i(x)\geq\lambda_1C_4(R/2)^d$
\end{enumerate}
\end{lema}
\begin{proof}

By Corollary \ref{cor:bounddens}, $X_{t_j}^i$ has a bounded density with respect to the Lebesgue measure. Using invariance by translation,
\begin{eqnarray}
\E{\kappa^{i,j}}^2_\theta(x)=\int_{\Pi^d}{\kappa^{i,j}}^2_\theta(x)\mu_{t_j}(y)dy
&\leq& \frac{\lambda_2\|\kappa\|_\infty}{\theta^{2d}}\int_{\Pi^d}\kappa\left(\frac{-y}{\theta}\right)dy\\
&\leq& \frac{\lambda_2\|\kappa\|_\infty}{\theta^{2d}}\int_{\Pi^d}\kappa\left(u\right)|(-\theta)^d|du=\frac{\lambda_2\|\kappa\|_\infty}{\theta^{d}}.
\end{eqnarray}
Note that we can rewrite this inequality as follows: 
\begin{equation}\label{ekappaquad}
    \E\kappa^2\left(\frac{x-X_{t_j}^i}{\theta}\right)\leq \lambda_2\|\kappa\|_\infty\theta^{d}.
\end{equation}
The second item is proven as follows:

\begin{align}
	\E\kappa_\theta^{i,j}(x)= \E\frac{1}{\theta^d}\kappa\left(\frac{x-X_{t_j}^i}{\theta}\right)&=\frac{1}{\theta^d}\int_{\Pi^d}\kappa\left(\frac{x-y}{\theta}\right)\mu_{t_j}(y)dy\\
	&\geq\frac{\lambda_1}{\theta^d}\int_{[-R/4,R/4]^d}\kappa\left(\frac{x-y}{\theta}\right)dy\\
	&\geq\frac{\lambda_1}{\theta^d}\int_{[-R/4,R/4]^d}\kappa\left(\frac{-y}{\theta}\right)dy\\
	&\geq\frac{\lambda_1\theta^d}{\theta^d}\int_{[-R/4,R/4]^d}\kappa\left(u\right)du\\
	&\geq\lambda_1 C_4 Leb([-R/4,R/4]^d) = \lambda_1C_4(R/2)^d.
	\end{align}
We can rewrite this inequality as follows: 
 \begin{equation}\label{ekappa}
     \E\kappa\left(\frac{x-X_{t_j}^i}{\theta}\right)\geq\lambda_1C_4(R/2)^d \theta^d.
 \end{equation}
 
\end{proof}
\begin{lema}\label{lem:smallprob}
Fix $x\in\Pi^d$ and consider the sets
\begin{equation*}
    E = \left\{\sum_{i=1}^N\frac{1}{\theta^d}\kappa\left(\frac{x-X_{t_j}^i}{\theta}\right)> \frac{\E Z_\theta(t_j,x)}{2}\right\} \text{ and } A = \left\{\frac{1}{N\theta^d}\sum_{i=1}^N\kappa\left(\frac{x-X_{t_j}^i}{\theta}\right)\geq \frac{\mu_{t_j}(x)}{2}\right\}.
\end{equation*}
Then 
\begin{enumerate}
    \item $$\mathbb{P}(E^C)\leq\exp\left(-N C_{\theta,\kappa}\right).$$
    
    \item Let $\theta$ be such that $0<\theta\lesssim\sqrt{\lambda_1}/\sqrt{2} $. Then
\begin{equation}\label{cond}
    \mathbb{P}(A^C)\leq\exp\left(-N C_{\theta,\kappa}\right).
\end{equation}
\end{enumerate}
\end{lema}
\begin{proof}

    We have that 
    \[E^C = \left\{\sum_{i=1}^N\frac{1}{\theta^d}\kappa\left(\frac{x-X_{t_j}^i}{\theta}\right)\leq \frac{\E Z_\theta(t_j,x)}{2}\right\} = \left\{\frac{1}{N}\sum_{i=1}^N(\kappa_\theta^{i,j}(x)-\E\kappa_\theta^{i,j}(x))\leq -\frac{\E\kappa_\theta^{i,j}(x)}{2}\right\}.\]

We can use Hoeffding's Inequality, Theorem 2.8 in \citep{boucheron2013concentration}, since $\kappa_\theta$ is a bounded random variable
\begin{align*}
    0\leq\kappa_\theta^{i,j}(x)\leq\frac{\|\kappa\|_\infty}{\theta^d}.
\end{align*}
and let $C_R$ be the constant given by Lemma \ref{lem:boundkappa}, first item. 
We conclude
\begin{align}
    \mathbb{P}\left(\frac{1}{N}\sum_{i=1}^N(\kappa_\theta^{i,j}(x)-\E\kappa_\theta^{i,j}(x))
    \leq -C_R\right)
    &\leq \exp\left(-\frac{2NC_R^2\theta^{2d}}{\|\kappa\|_\infty^2}\right)\\
    &\coloneqq\exp\left(-N C_{\theta,\kappa}\right).
\end{align}

The same reasoning applies to item 2., since we can rewrite
\[A^C = \left\{\frac{1}{N}\sum_{i=1}^N(\kappa_\theta^{i,j}(x)-\E\kappa_\theta^{i,j}(x))\leq \frac{\mu_{t_j}(x)}{2}-\E\kappa_\theta^{i,j}(x)\right\}.\]
The condition on $\theta$ guarantees that the right hand term is non-positive, since by Lemma \ref{lem:dens}, the bias term has order $\theta^2$, i.e,  $|\E\kappa_\theta^{i,j}(x)-\mu_{t_j}(x)|\leq C \theta^2$:
\begin{align}
    \frac{\mu_{t_j}(x)}{2}-\E\kappa_\theta^{i,j}(x)\leq -\frac{\mu_{t_j}(x)}{2}+C \theta^2\leq \frac{-\lambda_1}{2}+C\frac{\lambda_1}{2}<0.
\end{align}
We conclude

\begin{align}
    \mathbb{P}\left(A^C\right)\leq \exp\left(-N C_{\theta,\kappa}\right).
\end{align}
\end{proof}

\begin{lema}\label{lem:timedergradlogmu}
    For $x\in\Pi^d$ fixed, the function $\nabla\log\mu_t(x)$ is Lipschitz in time if Assumptions 2 and 3 holds.
\end{lema}
\begin{proof}
    Using the continuity equation and the boundedness assumptions on $\mu_t$ and $v_t$, we are going to prove that the time-derivative is bounded:
    \begin{align}
        \frac{d}{dt}\nabla\log\mu_t(x)= \frac{d}{dt}\left(\frac{\nabla \mu_t(x)}{\mu_t(x)}\right)& = \frac{\mu_t(x)d/dt(\nabla \mu_t(x))-\nabla \mu_t(x)\partial_t\mu_t(x)}{\mu_t^2(x)}\\
        &\lesssim d/dt(\nabla \mu_t(x))\\
        &\lesssim\nabla \partial_t\mu_t(x)\\
        & \lesssim\nabla \diver(\mu_t(x)v_t(x))\\
        & \lesssim \nabla^2\mu_t(x)v_t(x)+(\nabla\mu_t(x))^\top \nabla v_t(x)\\
        &+\nabla\mu_t(x)\diver v_t(x)+\mu_t(x)\nabla\diver v_t(x)\\
        &\leq C.
    \end{align}
    Note that we can exchange the order of time and space derivatives since the density $\mu$ is smooth in both variables, see \citep{mastereq} for a reference related to FPE and \citep{guillin2023uniformtimepropagationchaos} for a reference related to MVE. 

\end{proof}

		
\bibliography{references.bib}

\end{document}